\documentclass[11pt,reqno]{amsart}
\usepackage{algorithm,algorithmicx}
\usepackage{algpseudocode}
\usepackage{bbm}
\usepackage[mathlines,pagewise,left]{lineno}
\usepackage{amssymb}
\usepackage{color}
\usepackage{amsmath}
\usepackage{bcol}
\usepackage{multirow}
\usepackage[numbers]{natbib}
\bibpunct[, ]{(}{)}{,}{a}{}{,}

\usepackage{comment}
\usepackage[toc,page]{appendix}
\usepackage{graphicx} 
\usepackage{booktabs} 
\usepackage{units}
\usepackage{subcaption}
\usepackage{amsthm}
\usepackage{pdflscape}
\newtheorem{proposition}{Proposition}

\newtheorem{corollary}{Corollary}
\theoremstyle{definition}

\theoremstyle{remark}
\newtheorem{remark}{Remark}
\newtheorem{example}{Example}

\newcommand{\R}{\ensuremath{\mathbb{R}}}

\newcommand{\Z}{\ensuremath{\mathbb{Z}}}

\newcommand{\B}{\ensuremath{\mathbb{B}}}
\newcommand{\G}{\ensuremath{\mathbb{G}}}
\newcommand{\D}{\ensuremath{\mathbb{D}}}
\newcommand{\be}[1]{\begin{equation}\label{#1}}
\newcommand{\ee}{\end{equation}}

\DeclareMathOperator*{\argmin}{arg\,min}
\newcommand{\conv}{\ensuremath{\text{conv}}}
\newcommand{\diag}{\ensuremath{\text{diag}}}

\newcommand{\rev}[1]{{#1}}

\usepackage[normalem]{ulem}
\newcommand{\rep}[2]{{#1}{}}
\newcommand{\del}[1]{}
\newcommand{\ins}[1]{{#1}}

\makeatletter
\newcommand{\leqnomode}{\tagsleft@true}
\newcommand{\reqnomode}{\tagsleft@false}
\makeatother

\def\SingleSpacedXI{\linespread{1.05}}
\usepackage[colorinlistoftodos,prependcaption,textsize=small]{todonotes}
\title{ Submodularity in conic quadratic mixed 0-1 optimization}
\author{Alper Atamt\"urk and Andr\'{e}s G\'{o}mez}

\thanks{ \noindent \hskip -5mm
A. Atamt\"urk: Department of Industrial Engineering \& Operations Research, University of California, Berkeley, CA 94720.
\texttt{atamturk@berkeley.edu}   \\
A. Gomez: Department of Industrial Engineering, University of Pittsburgh, Pittsburgh, PA 15261. \texttt{agomez@pitt.edu}
}

\begin{document}

\maketitle

\begin{abstract}
We describe strong convex valid inequalities for conic quadratic mixed 0-1 optimization. 
\ins{These inequalities can be utilized for solving numerous practical nonlinear discrete optimization problems from value-at-risk minimization to queueing system design, from robust interdiction to assortment optimization through appropriate conic quadratic mixed 0-1 relaxations.}
The inequalities exploit the submodularity of the binary restrictions and are based on the polymatroid inequalities over binaries for the diagonal case. We prove that the convex inequalities completely describe the convex hull of a single conic quadratic constraint as well as the rotated cone constraint over binary variables and unbounded continuous variables. We then generalize and strengthen the inequalities by incorporating additional constraints of the optimization problem. Computational experiments on mean-risk optimization with correlations, assortment optimization, and robust conic quadratic optimization
 indicate that the new inequalities strengthen the convex relaxations substantially and lead to significant performance improvements. \\

\noindent
\textbf{Keywords:} Polymatroid, submodularity, second-order cone, nonlinear cuts, robust optimization, assortment optimization, value-at-risk, \rev{interdiction,} Sharpe ratio.

\end{abstract}

\begin{center}
October 2016; February 2018; August 2018 \\ \vskip 1cm 
\end{center}

\BCOLReport{16.02}

\pagebreak

\section{Introduction}
\label{sec:intro}

Submodular set functions play an important role in many fields and have received substantial interest in the literature as they can be minimized in polynomial time \citep{GLS:ellipsoid,S:combinatorial,O:faster}. Combinatorial optimization problems such as the min-cut problem, entropy minimization, matroids, binary quadratic function minimization with a non-positive matrix are special cases of submodular minimization \citep{F:submodularBook}.
The utilization of submodularity, however, has been mainly restricted to 0-1 optimization problems although many practical problems involve continuous variables as well. 

The goal in this paper is to exploit submodularity to derive valid inequalities for mixed 0-1 minimization problems with a conic quadratic objective:
\begin{align} \label{eq:general-obj}
\min  a'x + \Omega \sqrt{x'Qx}:  x \in X \subseteq \{0,1\}^n\times \R_+^m,
\end{align}
or a conic quadratic constraint:
\begin{align} \label{eq:general-const}
 a'x + \Omega \sqrt{x'Qx} \le \rev{r},  \  x \in X \subseteq \{0,1\}^n\times \R_+^m,
\end{align}
where $\Omega \in \R_{+}$\rev{, $r\in \R$} and $Q$ is a symmetric positive semidefinite matrix. Formulations~\eqref{eq:general-obj} and \eqref{eq:general-const} are frequently used to model mean-risk problems. In particular, \eqref{eq:general-obj}
is value-at-risk minimization and \eqref{eq:general-const} is a probabilistic constraint for a random
variable $\rev{\tilde p}' x$, with $\rev{\tilde p} \sim N(a, Q)$. They are also used to model conservative robust formulations with an appropriate value of $\Omega$ if $\rev{\tilde p}$ is not normally distributed \citep{RO-book}.

Introducing an auxiliary variable $z$ to represent
the square root term $\sqrt{x'Qx}$ in \eqref{eq:general-obj}--
\eqref{eq:general-const}, we write 
\begin{align*} 
f(x) = \sqrt{x'Qx} \le z,  \  x \in X \subseteq \{0,1\}^n\times \R_+^m.
\end{align*}
The motivation for this study stems from the fact that $f$ is submodular for the simplest nontrivial non-convex case: when $Q$ is diagonal and $m=0$ \citep{shen2003}.
Therefore, one may expect submodularity to play a significant role in analyzing and solving optimization problems with a general conic quadratic objective or constraint
as submodularity is contained in a basic form. 

Toward this goal we consider the conic quadratic mixed-binary set 
\begin{equation*}
\label{eq:defH}
H_X=\left\{(x,y)\in X, z\in \R_+: \sigma+\sum_{i=1}^n c_i x_i +\sum_{i=1}^m d_iy_i^2 \leq z^2 \right\},
\end{equation*}
where $X\subseteq \rev{\D} = \{0,1\}^n\times \R_+^m$, $c\in \R_+^n$, $d\in \R_+^m$ and $\sigma\geq 0$ and derive strong
inequalities for it. Note that
$H_\rev{\D}$ is the mixed-integer epigraph of the function
\[
f(x,y) =  \sqrt{\sigma+\sum_{i=1}^n c_i x_i +\sum_{i=1}^m d_iy_i^2}.
\]

The set $H_X$ arises frequently in mixed-integer optimization models, well beyond the natural extension to mixed 0-1 mean-risk minimization or chance constrained optimization with uncorrelated random variables. In particular, in Section~\ref{sec:applications} we describe applications on optimization with \emph{correlated} random variables, inventory and scheduling problems, assortment optimization, fractional linear binary optimization, Sharpe ratio maximization, facility location problems, and \rev{conic quadratic interdiction problems}. 

\ignore{
The set $H_X$ frequently arises in models of mixed-integer 
probabilistic optimization problems, scheduling and location problems, 
and ratio minimization problems as discussed in the next section. 
$H_X$ can also be used to derive valid
inequalities for the non-diagonal case through an appropriate relaxation (Section~\ref{sec:applications}). 
}

Let $H_\rev{\B}$ denote the pure binary case of $H_\rev{\D}$ with $m=0$, for which $f$ is submodular. 
While the convex hull of $H_\rev{\B}$, conv($H_\rev{\B}$), is a polyhedral set and well-understood, that is not the case for the mixed-integer set $H_\rev{\D}$. 
Note, however, that for a fixed $y$, $f$ is submodular in $x$. By exploiting this partial 
submodularity for the mixed-integer case, in this paper, we give a complete nonlinear inequality description of conv($H_\rev{\D}$). 
We review the polymatroid inequalities for the pure binary case  in Section~\ref{sec:binary}.

Moreover, we show that the resulting nonlinear inequalities are also strong for the rotated  conic quadratic mixed 0-1 set 
\begin{equation*}
\label{eq:defR}
R_X=\left\{(x,y)\in X, (w,z)\in \R_+^2:\sigma+ \sum_{i=1}^n c_i x_i+\sum_{i=1}^m d_iy_i^2\leq 4wz \right\} \cdot
\end{equation*}
\ignore{
which can be equivalently written as 
\begin{equation*}
\label{eq:defRX}
R_X=\left\{(x,y)\in \! X, (w,z)\in \R_+^2 \!: \sigma \! + \! \sum_{i=1}^n c_i x_i+ \! \sum_{i=1}^m d_iy_i^2+ \! (w \! - \! z)^2 \leq (w \! + \! z)^2 \right\}.
\end{equation*}
}
Observe that even for the binary case ($m=0$), the definition of $R_X$ has the product of two continuous variables $w,z$ on the right-hand-side.
Therefore, the existing polymatroid inequalities from the binary case cannot be directly applied to $R_X$. Several of the applications in Section~\ref{sec:applications} are modeled using the rotated cone set $R_X$.

\subsection*{Literature review}

\ignore{
In particular, the authors
use submodularity of the function
\begin{equation}
\label{eq:definitionF}
f(x)=\sqrt{\sigma+\sum_{i=1}^n c_i x_i}
\end{equation}  
to derive valid \emph{extended polymatroid inequalities} for $H_B$. The inequalities and bound constraints are sufficient to describe the convex hull of $H_B$.
}

A major difficulty in developing strong formulations for mixed-integer nonlinear sets such as $H_X$ is that the corresponding convex hulls are not polyhedral, while most of the theory and methodology developed for mixed-integer optimization focuses on the polyhedral case. Recently, there has been an increasing effort to generalize methods from the linear case to the nonlinear case, including Gomory cuts \citep{Cezik2005}, MIR cuts \citep{AN:conic-mir:ipco}, cut generating functions \citep{Santana2017}, minimal valid inequalities \citep{Kilincc2015}, conic lifting \citep{Atamturk2011}, intersection cuts, disjunctive cuts, and lift-and-project cuts \citep{CS:dis-conv,SM:lift-project}. \cite{Kilinc2010} and \cite{Bonami2011} discuss the separation of split cuts using outer approximations and nonlinear programming. Additionally, some classes of nonlinear sets have been studied in detail: \cite{Belotti2015} study the intersection of a convex set and a linear disjunction, \cite{Modaresi2014} study intersections of a quadratic and a conic quadratic inequalities, \cite{KilincKarzan2015} study disjunctions on the second order cone, \cite{Burer2017} study the intersection of a non-convex quadratic and a conic quadratic inequality, \cite{Dadush2011b} and \cite{Dadush2011c} investigate the the Chv{\'a}tal-Gomory closure of convex sets and \cite{Dadush2011} investigate the split closure of a convex set. These inequalities are general and do not exploit any special structure.

Another stream of research for mixed-integer nonlinear optimization involves generating strong cuts by exploiting structured sets as it is common for the linear integer case. Although the applicability of such cuts is restricted to certain classes of problems, they tend to be far more effective than the general cuts that ignore any problem structure.
\cite{akturk2009strong,aag:scheduling} give second-order representable perspective cuts 
for a nonlinear scheduling problem with variable upper bounds, which are generalized further by
\cite{Gunluk2010} \rev{and \cite{atamturk2018strong}}. \cite{aa:max-submodular} give strong lifted inequalities for maximizing a submodular concave utility function.
\cite{an:submodular-knapsack,ab:submodular-covering} study binary knapsack sets defined by a single second-order conic constraint. \cite{Modaresi2016} derive closed form intersection cuts for a number of structured sets. \cite{AJ:conicvub} give strong valid inequalities for mean-risk minimization with variable upper bounds.

Closely related to this paper,
\cite{Atamturk2008b} study $H_\rev{\B}$ in the context of mean-risk minimization.  \cite{Yu2015} study the generalization with a cardinality constraint, i.e., $H_Y$ where $Y=\left\{x\in \{0,1\}^n:\sum_{i=1}^nx_i\leq k \right\}$. However, more general sets have not been considered in the literature. More importantly perhaps, the valid inequalities derived for the pure-binary case have limited use for mixed-integer problems or even for pure-binary problems with correlated random variables (non-diagonal matrix Q).

\ignore{
In the current paper we derive strong cuts for the set $H_X$. Note that 
$$H_X=\left\{(x,y)\in X, z\in \R_+:\bar{f}(x,y)\leq z \right\},$$ where
$$\bar{f}(x,y)=\sqrt{\sigma+\sum_{i=1}^n c_i x_i+\sum_{i=1}^m d_iy_i^2}.$$
Observe that the non-convex constraint $\bar{f}(x,y)\leq z$ is stronger than the second order conic constraint given in \eqref{eq:defH} since $x_i^2<=x_i$ for $0\leq x_i\leq 1$, and equality only holds at $x_i=0$ and $x_i=1$. Function $\bar{f}$ is convex in $y$ but 
concave in $x$. In particular, for any fixed $y$, $\bar{f}$ is submodular in $x$. We show in this paper how to exploit the partial submodularity of $\bar{f}$ to obtain strong valid inequalities for $H_X$. 
}

\subsection*{Notation}
Let $x$ denote an $n$-dimensional vector of binary variables, $y$ denote an $m$-dimension vector of continuous variables, and $c$ and $d$ be \emph{nonnegative} vectors of dimension $n$ and $m$, respectively. Define $N=\{1,\ldots,n\}$ and $M=\{1,\ldots,m\}$. Let $\text{conv}(X)$ denote the convex hull of $X$. Given a vector $a\in \R^n$ and $S\subseteq\{1,\ldots,n\}$, let $\diag(a)$ denote the $n\times n$ diagonal matrix $A$ with $A_{ii}=a_i$, and let $a(S)=\sum_{i\in S}a_i$. Let $\rev{\B} = \{0,1\}^n$ and $\rev{\G}=\{0,1\}^n\times [0,1]^m$.

\subsection*{Outline}

The rest of the paper is organized as follows.
In Section~\ref{sec:applications} we discuss applications in which sets $H_X$ and $R_X$ arise naturally.
 In Section~\ref{sec:binary}  we review the 
 existing results for $H_\rev{\B}$ and $H_\rev{\G}$.  
 In Section~\ref{sec:unbounded} we show that a nonlinear generalization of the  polymatroid inequalities 
 is sufficient to describe the convex hull of $H_\rev{\D}$.
In Section~\ref{sec:bounded} we study the bounded set $H_\rev{\G}$, give an explicit convex hull description for the case $n=m=1$, and propose strong valid inequalities for the general case. In Section~\ref{sec:constrained} we describe a strengthening procedure for the  nonlinear polymatroid inequalities for any mixed-integer set $X$; the approach generalizes the lifting method of \cite{Yu2015} for the pure-binary cardinality constrained case. \rev{In Section~\ref{sec:compcons} we discuss the implementation of the proposed inequalities using off-the-shelf conic quadratic solvers.} In Section~\ref{sec:computational} we test the effectiveness of the proposed inequalities for a variety of problems discussed in Section~\ref{sec:applications}. Section~\ref{sec:conclusions} concludes the paper. 


\section{Applications}
\label{sec:applications}
In this section, we present \rev{seven} mixed 0-1 optimization problems in which sets $H_X$ and $R_X$ arise naturally.


\subsection{Mean-risk minimization and chance constraints with uncorrelated random variables}
\label{sec:meanrisk}

Conic quadratic constraints are frequently used to model probabilistic optimization with Gaussian distributions \citep[e.g.][]{BL:sp-book}. In particular, if $a_i$, $c_i$ denote the mean and variance of random variables $\tilde{p}_i$, $i\in N$, and $b_i$, $d_i$ the mean and variance of random variables $\tilde{q}_i$, $i\in \rev{M}$, and all variables are independent, then $$\min_{(x,y,z)\in H_X}a'x+b'y+\Phi^{-1}(\alpha) z$$ 
corresponds to the value-at-risk minimization problem over $X$, where $\Phi$ is the c.d.f. of the standard normal distribution and $0.5<\alpha<1$ . Alternatively, the chance constraint $\textbf{Pr}(\tilde{p}'x+\tilde{q}'y\leq \rev{r})\geq \alpha$ is equivalent to $a'x+b'y+\Phi^{-1}(\alpha) z\leq \rev r$, $(x,y,z)\in H_X$. Models with $H_X$ also arise in robust \rev{and distributionally robust} optimization problems with ellipsoidal uncertainty sets \citep{BenTal1998,BenTal1999,ben:robust-book,ghaoui2003worst,zhang2016ambiguous}.

\subsection{Mean-risk minimization and chance constraints with correlated random variables}
\label{sec:correlated} If $\tilde{p}\sim \mathcal{N}(a,Q)$, where $a$ is the mean vector and $Q\succeq 0$ is the covariance matrix, then the value-at-risk minimization or chance constrained optimization with 0--1 variables involve constraints of the form $\sqrt{x'Qx}\leq z$.

A standard technique in quadratic optimization consists in utilizing the diagonal entries of matrices to construct strong convex relaxations \citep[e.g.][]{poljak1995convex,anstreicher2012convex}. In particular, for $x \in \{0,1\}^n$, we have
\[
x'Qx \le z \iff x'(Q-\diag(c))x + c'x \le z
\]
with $c \in \R_+^n$ such that $Q-\diag(c) \succeq 0$.
This transformation is based on the ideal (convex hull) representation of the
separable quadratic term $x'\diag(c)x$ as $c'x$ for $x \in \{0,1\}^n$. Using a similar idea and introducing a continuous variable $y\in \R_+$, we get
\begin{equation*}
\label{eq:socpConstraint}
\sqrt{x'Qx}\leq z\Leftrightarrow(x,y,z)\in H_X \text{ and }\sqrt{x'(Q-\diag(c))x}\leq y.
\end{equation*}
The approach presented here can also be used for mixed-binary sets $X$.

\ignore{
Note that constraints of the form \eqref{eq:socpConstraint} correspond to general second order cone constraints. Thus the inequalities discussed in this paper are broadly applicable to SOCP-representable problems \citep{Alizadeh2003, Lobo1998}. We now discuss in Sections~\ref{sec:lotSizing}-\ref{sec:probabilityMax} additional classes of problems that can be modeled using $R_X$. 
}

\subsection{\rev{\ins{Robust} conic quadratic interdiction}}
\label{sec:robust}
\rev{Given a set of potential adverse event (e.g., natural disasters, disruptions, enemy attacks) \ins{scenarios} $C$, consider the problem of minimizing \rep{the}{a} worst-case cost where only \rep{a}{small} subset of the events can \rep{occur}{happen} simultaneously. If the nominal problem ---when no adverse \rep{event occurs}{happen}--- \rep{is}{can be formulated as} a mixed-integer linear optimization problem, then the worst-case minimization problem \rep{can be formulated as}{is} 
	\leqnomode
	\begin{align}
	\label{eq:linearRobust}\tag{LI}
	\min_{x\in X}\max_{u\in \mathcal{U}}a_0'x+\sum_{j\in C}(a_j'x)u_j,
	\end{align}
where $\mathcal{U}=\left\{u\in \{0,1\}^C: \sum_{j\in E}u_j\leq \Gamma\right\}$ is the uncertainty set, $\Gamma\in \Z_+$ is the maximum number of events \rep{that may occur simultaneously,}{allowed to happen} $a_0$ is the \rep{nominal cost vector}{vector of nominal costs} and $a_j\in \R_+^n$ is the \rep{additional cost vector}{of additional costs} if event $j$ \rep{occurs}{happens}. Problem \eqref{eq:linearRobust} arises naturally in robust optimization \citep{Bertsimas2003,Bertsimas2004}, and it has \del{also} received a vast amount of attention in the context of interdiction \citep[e.g.,][]{wood1993deterministic,cormican1998stochastic,israeli2002shortest,lim2007algorithms}.

We now consider the generalization, where the nominal problem is a mixed-integer conic quadratic optimization problem, e.g., \ins{with a } value-at-risk minimization \ins{objective, considered in \cite{ADJ:mr-interdiction}}. In this case, the worst-case minimization problem is
\leqnomode
\begin{align}
\label{eq:robustConicQuadraticProblem}\tag{CQI}
\quad\quad\omega^* =\min_{x\in X}\max_{u\in \mathcal{U}} a_0'x+\sum_{j\in C}\left(a_j' x\right) u_j+\sqrt{x'Q_0 x+\sum_{j\in C}\left(x'Q_j x\right)u_j},
\end{align}
where $Q_0\succeq 0$ is the nominal covariance matrix and $Q_j\succeq 0$ is the matrix of increased covariances if event $j$ happens.

Problem \eqref{eq:robustConicQuadraticProblem} was \del{first} studied by \cite{Atamturk2017} \rep{for a convex feasible set $X$}{when $X$ is a convex set}. \rep{They}{The authors} show that solving the inner maximization problem is \NP-hard for a fixed value of $x$ \rep{and that}{Nevertheless,} feasible solutions with objective values within 25\% of \ins{the} optimal can be obtained by solving the optimization \ins{problem}
\begin{align}
\omega_a=\min \;& \frac{1}{4}w+a_0'x+z_0+\Gamma\gamma \notag\\
\text{s.t.}\;& \gamma\geq  a_j' x+z_j & \forall j\in C \notag\notag\\
\label{eq:approxInterdiction} \tag{IA} 
&x'Q_jx\leq z_jw & \forall j\in \{0\} \cup C \notag \\
& x\in X,\; z\in \R_+^{|C|+1},\; w\in \R_+,\; \gamma\in \R_+.\notag
\end{align}
Formulations for the generalization where $\mathcal{U}$ is set of extreme points of an integral polytope are also proposed, but are omitted here for brevity. 

If the set $X$ is \rep{conic quadratic}{SOCP}-representable, then \del{problem} \eqref{eq:approxInterdiction} can be tackled with off-the-shelf \rep{mixed-integer conic quadratic}{SOCP} solvers. Moreover, if all $x$ variables are continuous, then \eqref{eq:approxInterdiction} is convex optimization problem, thus polynomial-time solvable. In contrast, if some variables are discrete, then \eqref{eq:approxInterdiction} is much more challenging, especially due to the rotated cone constraints $x'Q_jx\leq z_jw$. \del{ whose structure is not well understood to date. Nevertheless,} Observe that, \ins{in this case}, we can \del{use a similar ideas as the one presented in Section~\ref{sec:correlated},} introduce an additional variable $y\in \R_+$ and \rep{and then utilize}{use} the decomposition 
$$x'Q_jx\leq z_jw\Leftrightarrow (x,y,w,z_j)\in R_X\text{ and }x'(Q-\diag(c))x\leq y$$
}
\ins{to derive stronger formulations.}

\subsection{Lot-sizing and scheduling problems}
\label{sec:lotSizing}
Inventory problems with economic order quantity involve expressions of the form $k\frac{\rev p}{\rev q}$, where $\rev {p\in \R_+}$ is the demand, $\rev {q\in \R_+}$ is the lot size, and $k\rev{\in \R_+}$ is a fixed cost for ordering inventory. In simple settings, the optimal lot size $\rev q^*$ can be expressed explicitly \citep{Nahmias2001}, but in more complex settings, where the demand is a linear function of discrete variables, e.g., in joint location-inventory problems \citep{Ozsen2008,Atamturk2012} this is not possible. In such cases, the order costs involve expressions of the form
\begin{equation}
\label{eq:fractional1}
\frac{c'x}{\rev q}\leq z \Leftrightarrow (x,\rev q,z)\in R_{\rev \B}.
\end{equation}
The ratio \eqref{eq:fractional1} also arises in scheduling, specifically in the \emph{economic lot scheduling problem} \citep{Bollapragada1999,Bulut2014,Pesenti2003,Sahinidis1991}. In this context, $c$ is the vector to setup costs/times and $\rev q$ denotes a production cycle length, thus $z$ in \eqref{eq:fractional1} corresponds to setup costs/times per unit time. Expression \eqref{eq:fractional1} also arises in the plant design and scheduling problems to model the profitability or productivity of the plant \citep{Castro2005,Castro2009}.

\rev{
\subsection{Queueing system \ins{design}}
\label{sec:queueing} The service system design problem, also referred to as the facility location problem with stochastic demand and congestion \citep{amiri1997solution,berman200111,elhedhli2005exact,elhedhli2006service}, \rep{aims to locate}{is concerned with locating} a set of service facilities \ins{while} balancing operational costs and service quality. \del{costs.} If a \del{given} facility services \rep{too many}{a substantial amount of} customers, it may become overly congested, resulting in long waiting times for the customers and poor \ins{service} quality overall. 
Specifically, congestion is often modeled using queueing theory. 
Given an M/M/1 queue with mean demand $\lambda$ and mean service rate $\mu>\lambda$, the average time in the system is $\frac{1}{\mu-\lambda}$. Additionally, in the service system design problem, the demand at location $i$ is of the form $\lambda_i=c_i'x$, where $x$ are binary decision variables modeling \ins{the} assignments of customers to facilities; moreover, the service rates are of the form $\mu_i=d_i'y$, where $y$ are variables representing the \del{ type of}servers installed at location $j$. Thus the service system design problem is of the form
\leqnomode
\begin{align*}\label{eq:ssdp}\tag{SSDP}\min_{(x,\ins{y},t)\in X}a'x+b'y+\Omega\sum_{i}\frac{c_i'x}{a_i'y-c_i'x},\end{align*}
where $\Omega>0$ is the weight given to the service quality, and each term $\frac{c_i'x}{a_i'y-c_i'x}$ is the total time of servicing \rep{the}{all} customers at location $j$. Observe that 
$$\frac{c_i'x}{a_i'y-c_i'x}\leq z \Leftrightarrow (x,\mu-\lambda,z)\in R_{\B},$$
thus strong formulations for $R_{\B}$ can be directly used in the context of \eqref{eq:ssdp}. 
}

\subsection{Binary linear fractional problems }
\label{sec:fractional}

Generalizing the models in Section\rev{s}~\ref{sec:lotSizing} \rev{and \ref{sec:queueing}}, binary linear fractional problems are optimization problems with constraints of the form
\begin{align*}
\label{eq:fractional01}
 \frac{ c_{0}+\sum_{i=1}^nc_{i}x_i}{a_{0}+\sum_{i=1}^n a_{i}x_i}\leq z &\Leftrightarrow c_0+\sum_{i=1}^n c_ix_i^2\leq zw,\ w=a_{0}+\sum_{i=1}^n a_{i}x_i\\
 &\Leftrightarrow (x,w,z)\in R_B\text{, with }w=a_{0}+\sum_{i=1}^n a_{i}x_i
\end{align*}
where $a_{i}, c_{i}\geq 0$ for $i=0,\ldots,n$. Note that a lower bound on the ratio can also be expressed similarly
by complementing variables.
\ignore{
 by noting that
$$z \leq  \frac{ c_{0}+\sum_{i=1}^nc_{i}x_i}{a_{0}+\sum_{i=1}^n a_{i}x_i} \Leftrightarrow z \leq L-w\text{ with } \frac{ (La_0-c_{0})+\sum_{i=1}^n(La_i-c_{i})x_i}{a_{0}+\sum_{i=1}^n a_{i}x_i}\leq w,$$
for any constant $L$; thus by selecting $L$ sufficiently large and complementing variables with $b_i=0$, we find an equivalent formulation with upper bounds on the ratios.
}
Binary fractional optimization arises in numerous applications including assortment optimization with mixtures of multinomial logits \citep{desir2014near,Mendez2014,Sen2015}, WLAN design \citep{Amaldi2011}, facility location problems with market share considerations \citep{Tawarmalani2002}, and cutting stock problems \citep{Gilmore1963}, among others; see also the survey \cite{Borrero2016} and the references therein.

Applications of binary linear fractional optimization are abundant in network problems. For example, given a graph $G=(V,E)$, problems of the form 
\begin{equation}
\label{eq:cutRatio}
\min\! \left\{\! \frac{\sum_{(i,j)\in E}c_{ij}x_{ij}}{\sum_{i\in V}a_i\rev{x}_i}: x_{ij} \!
\geq |\rev{x}_i \!- \!\rev{x}_j|, (i,j) \! \in \! E, \rev{x}\in \rev{X} \!\subseteq\{0,1\}^{\rev {V+E}}\right\}
\end{equation}
arise in the study of expander graphs \citep{Davidoff2003}; in particular, the optimal value of \eqref{eq:cutRatio} with $c=1$, $a=1$ and $\rev{X}=\left\{\rev{x}\in \{0,1\}^{V+\rev{E}}: 1\leq \sum_{i\in V}\rev{x}_i\leq 0.5|V|\right\}$ corresponds to the Cheeger constant of the graph.  
See \cite{Hochbaum2010,Hochbaum2013} for other fractional cut problems arising in image segmentation, and see \cite{Prokopyev2009} for a discussion of other ratio problems in graphs arising in facility location.

\subsection{Sharpe ratio maximization}
\label{sec:probabilityMax} Let $a_i,c_i$ be the mean and variance of normally distributed independent random variables $\tilde{p}_i$, $i\in N$ as in Section~\ref{sec:meanrisk}. A natural alternative to mean-risk minimization for a risk-adverse decision maker is, given a budget $\rev{r}$, to maximize the probability of meeting the budget; that is,
\begin{equation}
\label{eq:maxProba}
\max_{x\in X}\textbf{Pr}\left(\tilde{p}'x\leq \rev{r}\right).
\end{equation}
Problems of the form \eqref{eq:maxProba} are considered in \cite{Nikolova2006} in the context of the stochastic shortest path problem. 

Assuming there is a solution $x \in X$ satisfying $a'x\leq \rev{r}$, note that  
$$\textbf{Pr}\left(\tilde{p}'x\leq \rev{r}\right)=\textbf{Pr}\left(\frac{\tilde{p}'x-a'x}{\sqrt{c'x}}\leq \frac{\rev{r}-a'x}{\sqrt{c'x}}\right)=\Phi\left( \frac{\rev{r}-a'x}{\sqrt{c'x}}\right) \cdot$$
Since $\Phi$ is monotone non-decreasing and $\rev{r}-a'x\geq 0$ for any optimal solution, we see that \eqref{eq:maxProba} is equivalent to maximizing $\frac{\rev{r}-a'x}{\sqrt{c'x}}$. Observe that the resulting objective corresponds to maximizing the reward-to-volatility or Sharpe ratio \citep{Sharpe1994}, a commonly used risk-adjusted performance measure in finance. 
Maximizing the Sharpe ratio is equivalent to minimizing $\frac{\sqrt{c'x}}{\rev{r}-a'x}$. 
Therefore, we can restate \eqref{eq:maxProba} as 
\begin{align}
\min\;& z\notag\\
\text{s.t.}\;& w=\rev{r}-a'x\notag\\
&{\sqrt{c'x}}\leq wz\label{eq:pOrderConstraint}\\
& x\in X,\; w,z\geq 0.\label{eq:pOrderNonnegativity}
\end{align}
Constraint \eqref{eq:pOrderConstraint} is not conic quadratic. Note, however, for $w,z\geq 0$ we have $${\sqrt{c'x}}\leq wz \Leftrightarrow \sqrt{4\left(\sqrt[\uproot{3}4]{c'x}\right)^2+(w-z)^2}\leq w+z.$$ Then one gets a convex relaxation by replacing the non-convex term $\sqrt[\uproot{3}4]{c'x}$ by its convex lower bound $\sqrt[\uproot{5}4]{\sum_{i\in N}c_ix_i^4}$. The resulting conic quadratic representable convex inequality can be written as $${\sqrt{\sum_{i\in N}c_ix_i^4}}\leq wz.$$


As we will show in Section~\ref{sec:rotated}, a nonlinear version of the extended polymatroid inequalities corresponding to the submodular function $\bar h(x)=2\sqrt[\uproot{3}4]{c'x}$ is sufficient to describe the convex hull of the set given by \eqref{eq:pOrderConstraint}--\eqref{eq:pOrderNonnegativity} for $X= \rev{\B}$
(Remark~\ref{remark:ratio}).

\section{Preliminaries}
\label{sec:binary}


In this section we review \rev{\rep{earlier}{existing} results} for the binary and mixed 0-1 cases.
Given $\sigma \ge 0$ and $c_i > 0, \ i \in N$,  consider the set
\begin{align}
H_\rev{\B} = \bigg \{ (x,z) \in B \times \R_+:  \sqrt{\sigma + \sum_{i \in N} c_i x_i } \le z  \bigg \}.
\end{align}
Observe that $H_\rev{\B}$ is the binary restriction of $H_\rev{\D}$ obtained by setting $y = 0$ and it is the union of finite number \rev{of} line segments; therefore, its convex hull is polyhedral.
For a given permutation $\left((1),(2),\ldots,(n)\right)$ of $N$, let
\begin{align}
\sigma_{(k)}&= c_{(k)} + \sigma_{(k-1)}, \text{ and } \sigma_{(0)} = \sigma,\notag\\
\pi_{(k)}&=\sqrt{\sigma_{(k)}}-\sqrt{\sigma_{(k-1)}},\label{eq:definitionPi}
\end{align}
and define the \emph{polymatroid inequality} as
\begin{equation}
\label{eq:extendedPolymatroidInequality}
\sum_{i=1}^n \pi_{(i)}x_{(i)}\leq z - \sqrt{\sigma}.
\end{equation}
Let $\Pi_\sigma$ be the set of such coefficient vectors $\pi$ for \textit{all} permutations of $N$.
The set function defining $H_B$ is non-decreasing submodular; therefore, $\Pi_\sigma$ \rev{is the set of} extreme points of \rev{the extended} polymatroid \citep{Edmonds1970} \rev{associated with the submodular function $f(x)=\sqrt{\sigma+\sum_{i\in N}c_ix_i}$}. \rev{Then it follows from \citet{L:submodular-convex} that}
the convex hull of $H_\rev{\B}$ is given by
the set of all polymatroid inequalities and the bounds of the variables\rev{:}

\begin{proposition}
	[Convex hull of $H_B$]
	\label{prop:convexHullHB}
	$$\text{conv}(H_B)=\left\{(x,z)\in [0,1]^N\times \R_+:\pi'x \leq z - \sqrt{\sigma}, \;\; \forall \pi\in \Pi_\sigma \right\}.$$
\end{proposition}

\rev{Proposition~\ref{prop:separation} is a direct consequence of a result by \cite{Edmonds1970}, showing} the maximization of a linear function over a polymatroid can be solved by the greedy algorithm. Therefore, a point 
$(\bar x, \bar z) \in [0,1]^N \times \R_+$ can be separated from $\text{conv}(H_{\rev{\B}})$ via the greedy algorithm by sorting $\bar x_i, \ i \in N$ in non-increasing order in $O(n \log n)$ time.

\begin{proposition}
	[Separation]
	\label{prop:separation}
	A point $(\bar{x},\bar{z})\not \in \text{conv}(H_B)$ such that $\bar{x}_{(1)}\geq \bar{x}_{(2)} \geq \cdots \geq \bar{x}_{(n)}$ is separated from $\conv(H_{\rev\B})$ by inequality \eqref{eq:extendedPolymatroidInequality}.
\end{proposition}

\citet{Atamturk2008b} consider the  mixed-integer version of $H_\rev{\B}$:
\begin{equation*}
H_{\rev{\G}}=\left\{(x,y,z)\in C \times \R_+: {\sqrt{\sigma + \sum_{i\in N}c_ix_i+\sum_{i\in M}d_iy_i^2}}\leq z\right\},
\end{equation*}
where $d_i > 0, \ i \in M$, and give valid inequalities for $H_{\rev{\G}}$ based on 
the polymatroid inequalities.
Without loss of generality, the upper bounds of the continuous variables in $H_{\rev{\G}}$ are set to one by scaling.
\begin{proposition}
	[Valid inequalities for $H_{\rev{\G}}$]
	\label{prop:validIneqBounded1}
	For $T\subseteq M$ inequalities
	\begin{equation}
	\label{eq:polymatroidBoundedDominated}
	\pi'x +	\sqrt{\sigma + \sum_{i\in T}d_iy_i^2} \leq z, \quad \pi\in \Pi_{\sigma + d(T)}
	\end{equation}
	are valid for $H_{\rev{\G}}$.
\end{proposition}


Inequalities \eqref{eq:polymatroidBoundedDominated} are obtained by setting the subset $T$ of the continuous variables to their upper bounds and relaxing the rest, and they dominate any inequality of the form $$\xi'x + \sqrt{\sigma+\sum_{i\in T}d_iy_i^2}\leq z$$ with $\xi\in \R^n$. Although inequalities \eqref{eq:polymatroidBoundedDominated} are the strongest possible among inequalities that are linear in $x$ and conic quadratic in $y$, they may be weak or dominated by other classes of nonlinear inequalities.
In this  paper we introduce stronger and more general inequalities than \eqref{eq:polymatroidBoundedDominated} for $H_{\rev{\G}}$.

\ignore{
	**************************
	
	\subsection{Previous work}
	We now review previous results for the set $H_X$.
	
	\subsubsection{Case with $X=\{0,1\}^n$}
	
	In this section we state, without proof, the main results of \cite{Atamturk2008b} for the pure-binary set $H_B$. Note that since $H_B$ is a finite union of line segments, $\text{conv}(H_B)$ is a polyhedron, i.e., it can be described by a finite number of linear inequalities. As noted before, the function $f$ given in \eqref{eq:definitionF} is submodular since it is the composition of a linear function and a concave function. Let the polyhedron
	$$P(f)=\left\{\pi\in \R^n: \sum_{i\in N}\pi_i x_i\leq f(x)-\sqrt{\sigma}, \forall x\in\{0,1\}^n \right\}$$
	be the extended polymatroid associated with $f$, and let $\Pi(f)$ denote the set of extreme points of $P(f)$. Then,
	
	\begin{proposition}[Convex hull of $H_B$]
		\label{prop:convexHullHB}
		\begin{equation}
		\label{eq:extendedPolymatroidInequality}
		\text{conv}(H_B)=\left\{(x,z)\in [0,1]^n\times \R_+:\sqrt{\sigma}+\pi'x \leq z, \;\; \forall \pi\in \Pi(f)\right\}.
		\end{equation}
	\end{proposition}
	
	In this paper, extended polymatroid are always defined with respect to a function of the form $f_\sigma(x)=\sqrt{\sigma+c'x}$. Thus, for simplicity, we will write $\Pi_\sigma$ instead of $\Pi(f_\sigma)$, and we will write $\Pi$ instead of $\Pi_0$.
	
	\cite{Edmonds1970} gives a characterization of the elements of $\Pi_\sigma$. In particular, $\pi\in \Pi_\sigma$ if, for a permutation $\left((1),(2),\ldots,(n) \right)$,
	\begin{align}
	\pi_{(k)}&=\sqrt{c_{(k)} +\sigma_{(k)}}-\sqrt{\sigma_{(k)}}, \text{ where}\label{eq:definitionPi}\\
	\sigma_{(k)}&=\sigma+\sum_{i=1}^{k-1}c_{(i)}.\notag
	\end{align}
	Moreover, Edmonds also shows that optimization over extended polymatroids can be solved by the greedy algorithm. Thus, 
	a point $\bar x \in \R_+^n$ can be separated from $\text{conv}(H_B)$ via the greedy algorithm by sorting $\bar x_i$ in non-increasing order
	in $O(n \log n)$.

	\begin{proposition}[Separation]
		\label{prop:separation}
		A point $\bar{x}\not \in \text{conv}(H_B)$ such that $\bar{x}_{(1)}\geq \bar{x}_{(2)} \geq \ldots \geq \bar{x}_{(n)}$ is separated from $\conv(H_B)$ by the inequality $\sqrt{\sigma}+\pi'x\leq z$, where $\pi$ satisfies \eqref{eq:definitionPi}.
	\end{proposition}

	\subsubsection{Case with $X=\{0,1\}^n\times [0,1]^m$}

	\cite{Atamturk2008b} also give valid inequalities for the mixed-binary set with bounded continuous variables $H_X$.
	The assumption that the upper bound is $1$ is without loss of generality, since otherwise the continuous variables can be scaled.
	\begin{proposition}
		\label{prop:validIneqBounded1}
		For $T\subseteq M$, inequalities
		\begin{equation}
		\label{eq:polymatroidBoundedDominated}
		\sqrt{\sigma+\sum_{i\in T}d_iy_i^2}+\pi'x\leq z, \quad \pi\in \Pi_{\sigma+d(T)}
		\end{equation}
		are valid for $H_X$, where $X=\{0,1\}^n\times[0,1]^m$.
	\end{proposition}

	Inequalities \eqref{eq:polymatroidBoundedDominated} are obtained by setting the subset $T$ of the continuous variables to their upper bounds and relaxing the rest, and they dominate any inequality of the form $$\sqrt{\sigma+\sum_{i\in T}d_iy_i^2}+\xi'x\leq z$$ with $\xi\in \R^n$. Although inequalities \eqref{eq:polymatroidBoundedDominated} are the strongest possible among inequalities that are linear in $x$ and conic quadratic in $y$, they may be weak or dominated by other classes of nonlinear inequalities.

	*******************
}

\section{The case of unbounded continuous variables}
\label{sec:unbounded}

In this section we focus on the case with unbounded continuous variables, i.e., on $H_\rev{\D}$, where $\rev{\D}= \{0,1\}^n \times \R_+^n$. 
In this case, since the continuous variables have no upper bound,  
the only class of valid inequalities of type \eqref{eq:polymatroidBoundedDominated} are the polymatroid inequalities
\begin{equation}
\label{eq:binaryPolymatroidCuts}
\sqrt{\sigma}+\pi'x \leq z, \;\; \forall \pi\in \Pi_\sigma
\end{equation}
themselves from the ``binary-only" relaxation by letting $T = \emptyset$. Inequalities \eqref{eq:binaryPolymatroidCuts} ignore the continuous variables and are, consequently, weak for $H_\rev{\D}$. Here, we define a new class of \textit{nonlinear} valid inequalities and prove that they are sufficient to define the convex hull of $H_\rev{\D}$.

Consider the inequalities
\begin{equation}
\label{eq:polymatroidUnbounded}
{\left(\sqrt{\sigma}+\pi'x\right)^2 +\sum_{i\in M}d_iy_i^2}\leq z^2,\quad \pi\in \Pi_\sigma.
\end{equation}

\ignore{
****************** SHOULD BE AS FOLLOWS ******************

Let $(\cdot)^+$ denote $\max\{0, \cdot\}$.

Consider the inequalities
\begin{equation}
\label{eq:polymatroidUnbounded}
\sqrt{\left((\sqrt{\sigma}+\pi'x)^+\right)^2 +\sum_{i\in M}d_iy_i^2}\leq z,\quad \pi\in \Pi_\sigma.
\end{equation}

as the auxiliary variable $s \ge 0$.

**********************************************************
}

\begin{proposition}
\label{prop:validityUnbounded}
Inequalities \eqref{eq:polymatroidUnbounded} are valid for $H_\rev\D$.
\end{proposition}
\begin{proof}
Consider the extended formulation of $H_\rev\D$ given by
\begin{equation*}
\widehat{H}_\rev\D=\left\{(x,y)\in \rev\D, (z,s)\in\R_+^{2}: {s^2 +\sum_{i\in M} d_i y_i^2}\leq z^2, {\sigma+\sum_{i\in N}c_ix_i}\leq s^2\right\}.
\end{equation*}
The validity of inequalities \eqref{eq:polymatroidUnbounded} for $H_\rev\D$ follows directly from the validity of
the polymatroid inequality $\sqrt{\sigma}+\pi'x\leq s$, $\pi\in \Pi_\sigma$ (Proposition~\ref{prop:convexHullHB}) for $\widehat{H}_\rev\D$.
\end{proof}

\begin{remark}
 For $M=\emptyset$ inequalities \eqref{eq:polymatroidUnbounded} reduce to the  polymatroid inequalities \eqref{eq:extendedPolymatroidInequality}.
\end{remark}
	\ignore{
\begin{remark}
\label{rem:extended}
Although inequalities \eqref{eq:polymatroidUnbounded} are nonlinear in the original space of variables, they can be represented as linear inequalities in the extended formulation $\widehat{H}_\rev\D$. Such a linear representation is desirable when using \eqref{eq:polymatroidUnbounded} as cutting planes in branch-and-cut algorithms.
\end{remark}}

\begin{remark}
Since inequalities \eqref{eq:polymatroidUnbounded} correspond to polymatroid inequalities in an extended formulation, the separation for them is the same as in the binary case and can be done by sorting in $O(n\log n)$ (Proposition~\ref{prop:separation}).
\end{remark}

Inequalities \eqref{eq:polymatroidUnbounded} are obtained simply by extracting a submodular component from function ${f}$. The approach can be generalized to sets of the form \rev{
\begin{equation*}
U=\left\{x\in X, (y,z)\in \R_+^{m+1}: {h(x) +\sum_{i\in M} d_i y_i^2}\leq z^2\right\},
\end{equation*}
and $h:\{0,1\}^n\to \R_+$ is an arbitrary nonnegative function. Define $$U_0=\left\{x\in X, s\geq 0: \sqrt{h(x)}\leq s\right\}$$
and} observe that since $\rev{U_0}$ is a finite union of line segments, $\text{conv}(\rev{U_0})$ is a polyhedron. Moreover, valid inequalities for $\conv(\rev{U_0})$ of the form $\rev{\xi}'x\leq s$, $\rev{\xi \in \Xi}$, can be lifted into valid nonlinear inequalities for $U$ of the form 
\begin{equation}
\label{eq:nonlinearU}
{(\rev{\xi}'x)^2+\sum_{i\in M}d_iy_i^2}\leq z^2.
\end{equation} 
Proposition~\ref{prop:convexHullU} below implies inequalities of the form \eqref{eq:nonlinearU} are sufficient to describe $\text{conv}(U)$ if $\rev{\xi}'x\leq s$, $\rev{\xi \in \Xi}$, are
sufficient to describe $\text{conv}(\rev{U_0})$.

\ignore{
Inequalities \eqref{eq:polymatroidUnbounded} are obtained simply by extracting a submodular component from function ${f}$. The approach can be generalized to sets of the form 
\begin{equation*}
U=\left\{(x,s)\in K, (y,z)\in \R_+^{m+1}: {s^2 +\sum_{i\in M} d_i y_i^2}\leq z^2\right\},
\end{equation*}
where $K=\left\{x\in \bar{X}\subseteq\{0,1\}^n, s\geq 0: \sqrt{h(x)}\leq s\right\}$ and $h:\{0,1\}^n\to \R_+$ is an arbitrary nonnegative function. Observe that since $K$ is a finite union of line segments, $\text{conv}(K)$ is a polyhedron, and valid inequalities for $K$ of the form $\gamma'x\leq s$,  $\gamma \in \Gamma$, can be lifted into valid nonlinear inequalities for $U$ of the form 
\begin{equation}
\label{eq:nonlinearU}
{(\gamma'x)^2+\sum_{i\in M}d_iy_i^2}\leq z^2.
\end{equation} 
Proposition~\ref{prop:convexHullU} below implies inequalities of the form \eqref{eq:nonlinearU} are sufficient to describe $\text{conv}(U)$ if $\gamma'x\leq s$, $\gamma \in \Gamma$, are
sufficient to describe    $\text{conv}(K)$.
}

\begin{proposition}
\label{prop:convexHullU}
The convex hull of $U$ is described as
$$\text{conv}(U)=\rev{\left\{(x,y,z)\in \R_+^{n+m+1}:\exists s \text{ s.t. }(x,s)\in \conv(\rev{U_0}) \text{ and } s^2 +\sum_{i\in M} d_i y_i^2\leq z^2 \right\}}\cdot$$
\end{proposition}
\begin{proof}
{
Consider the optimization of an arbitrary linear function over \rev{the extended formulation of $U$ obtained by adding a variable $s\geq 0$ and the constraint $\sqrt{h(x)}\leq s$,}
\begin{align*}
\ \ \ \ \ \ \ \min\; & -a'x -b'y +r z\notag\\
\text{(BP)} \ \ \ \ \ \text{s.t.}\;& {s^2 +\sum_{i\in M}d_i y_i^2}\leq z^2, \ (x,s)\in \rev{U_0},  \ y\in \R_+^m, z\geq 0\notag\notag
\end{align*}
}
and over its convex relaxation,
\begin{align*}
\ \ \ \ \ \ \ \min\; & -a'x -b'y +r z\notag\\
\text{(P1)} \ \ \ \ \ \text{s.t.}\;& {s^2 +\sum_{i\in M}d_i y_i^2}\leq z^2,
(x,s)\in \text{conv}(\rev{U_0}), \  y\in \R_+^m, z\geq 0.\notag 
\end{align*}

We prove that for any linear objective both (BP) and (P1) are unbounded or (P1) has an optimal solution that is integer in $x$.
Without loss of generality, we can assume that $r> 0$ (if $r<0$ then both problems are unbounded, and if $r=0$ then (P1) reduces to a linear program over an integral polyhedron by setting $z$ sufficiently large, and is equivalent to (BP)), $r=1$ (by scaling), $b_i>0$ (otherwise $y_i=0$ in any optimal solution), and $d_i=1$ for all $i \in M$ (by scaling $y_i$). 

Eliminating the variable $z$ from (P1) we restate the problem as
\begin{align*}
 \ \ \text{(P2)} \ \ \ \  \min \left \{ -a'x -b'y + \sqrt{s^2 +\sum_{i\in M} y_i^2} \text{ : }\;
 (x,s)\in \text{conv}(\rev{U_0}), \
 y\in \R_+^m \right \} \cdot
\end{align*}
Note that if $y=0$ in an optimal solution of (P2), then (P2) reduces to a linear optimization over $\text{conv}(\rev{U_0})$, which has an optimal integer solution. Thus we assume that $\sqrt{s^2 +\sum_{i\in M} y_i^2}>0$, and in that case the objective function is differentiable and, by convexity of (P2), optimal solutions correspond to KKT points.
Let $\mu \in \R_+^m$ be the dual variables for constraints $y \ge 0$.
From the KKT conditions of (P2) with respect to $y$, we see that
\begin{align*}
-\mu_k &=b_k-\frac{y_k}{\sqrt{s^2+\sum_{i\in M}y_i^2}}, \  \forall k\in M.
\end{align*}
However, the complementary slackness conditions $y_k \mu_k = 0$ imply that $\mu_k = 0$ for all $k$, as otherwise $-\mu_k = b_k$ contradicts the assumption that $b_k > 0$. Therefore, it holds that
\begin{align*}
 y_k&=b_k\sqrt{s^2+\sum_{i\in M} y_i^2}, \ \ \forall k\in M.
\end{align*}
Defining $\beta=\sum_{i=1}^m b_i^2$, we have
\begin{align*}
\sum_{i\in M} b_i y_i=\beta\sqrt{s^2+\sum_{i\in M} y_i^2}
\end{align*}
and
\begin{align}
\label{eq:ySquared}
\sum_{i\in M} y_i^2&=\beta\left(s^2+\sum_{i\in M}y_i^2\right) .
\end{align}
Observe that if $\beta\geq 1$, equality \eqref{eq:ySquared} cannot be satisfied (unless $\beta=1$ and $s=0$), and the feasible (P2) is dual infeasible. {Indeed, let $\lambda>0$ and $\bar{y}_i=\lambda b_i$ for all $i\in M$,
and observe that for any value of $s$ $$\lim_{\lambda\to \infty}-b'\bar{y}+\sqrt{s^2+\sum_{i\in M}\bar{y}_i^2} =\begin{cases}-\infty & \text{if }\beta>1 \\ 0 &\text{if }\beta=1.\end{cases}$$ 
Thus, if $\beta>1$, then both 	problems (BP) and (P2) are unbounded. Moreover, if $\beta=1$, let $$(x^*,s^*)\in \argmin\limits_{(x,s)\in \text{conv}(\rev{U_0})}-a'x $$ with minimal value of $s^*$; if $s^*=0$, then $(x^*,\bar{y},s^*)$ is an optimal solution of both (BP) and (P2) for any $\lambda>0$, and if $s^*>0$ then there does not exist an optimal solution for problems (BP) and (P2), but infima of the objective functions are attained at $x^*$, $s^*$ and $y=\bar{y}$ as $\lambda\to \infty$.} 

If $\beta<1$, then we deduce from \eqref{eq:ySquared} that
\begin{align*}
\sum_{i\in M} y_i^2&=\frac{\beta}{1-\beta}s^2.
\end{align*}
Replacing the summands in the objective, we rewrite (P2) as
\begin{align}
\ \ \ \ \ \  \min\; & -a'x +s\rev{\sqrt{1-\beta}}  \notag\\
\text{(P3)} \ \ \ \ \ \ \text{s.t.}\;
&  (x,s)\in \text{conv}(\rev{U_0})\notag.
\end{align}
As $\beta < 1$, (P3) has an optimal solution and it is integral in $x$. \rev{By projecting out the additional variable $s$, we obtain the desired result.}
\end{proof}

\begin{remark}
From Proposition~\ref{prop:convexHullU} we see that, with no constraints on the continuous variables, describing the mixed-integer set $\text{conv}(H_X)$ reduces to describing a polyhedral set. Moreover, strong inequalities from pure binary sets \citep[e.g.,][]{Yu2015} can be naturally lifted into strong inequalities for $H_X$.
\end{remark}

\begin{corollary}
\label{cor:convexHullH}
Inequalities \eqref{eq:polymatroidUnbounded} and bound constraints completely describe $\text{conv}(H_\rev\D)$.
\end{corollary}
\begin{proof}
Follows from Proposition~\ref{prop:convexHullU} with $\rev{U_0}=H_\rev{\B}$, where the convex hull of $H_\rev{\B}$ is given in Proposition~\ref{prop:convexHullHB}, and substituting out the auxiliary variable $s$. 
\end{proof}

\subsection{Comparison with inequalities in the literature}
\label{sec:case1Var}

As seen in this section inequalities \eqref{eq:polymatroidUnbounded} give the convex hull of $H_\rev\D$. Therefore, they are the strongest possible inequalities for $H_\rev\D$. It is of interest to study the relationships to inequalities previously given in the literature. It turns out that for the case of a single binary variable, they can be obtained as either split cuts or conic MIR inequalities based on a single disjunction. The equivalence does not hold in higher dimensions\rev{, as in such cases $H_\rev\D$ is a disjunction of $2^n$ sets and } neither split cuts nor conic MIR inequalities \rev{based on single disjunctions} are sufficient to describe $\conv(H_\rev\D)$.

To see the equivalence, we now consider the special case of conic quadratic constraint with a single binary variable $x$:
$$H^1=\left\{(x,y,z)\in \{0,1\}\times\R_+^{m+1}:\sqrt{\sigma+cx+\sum_{i\in M} d_iy_i^2} \leq z\right\}.$$

\subsubsection{Comparison with split cuts}
We first compare inequalities~\eqref{eq:polymatroidUnbounded} with the split cuts given in \cite{Modaresi2016}. Following the notation used by the authors, let $$B=\left\{(y,z)\in \R_+^{m+2}:\sqrt{\sigma+y_0^2+\sum_{i\in M} d_iy_i^2} \leq z\right\}$$ be the base set, let $F=\left\{y\in \R_+^{m+1}: 0\leq y_0\leq c\right\}$ be the forbidden set, and define $K=B\setminus\text{int}(F)$, where $\text{int}(F)$ denotes the interior of $F$. Letting $y_0:= \sqrt{c}x$, we see that $H^1$ and $K$ are equivalent.

First consider the case $\sigma=0$. From Corollary~\ref{cor:convexHullH} we see that that $$\text{conv}(H^1)=\left\{(x,y,z)\in [0,1]\times\R_+^{m+1}: \sqrt{cx^2+\sum_{i\in M} d_iy_i^2}\leq z\right\}.$$ 
Moreover, from Corollary 5 of \cite{Modaresi2016}, since $0\not\in (0,c)$, we find that $\text{conv}(K)=B$. Thus, the results coincide in that the convex hulls of $H^1$ and $K$ are the natural convex relaxations of the sets.

Now consider the case $\sigma>0$. From Corollary~\ref{cor:convexHullH} we see that that 
\begin{equation}
\label{eq:convexHullH1Sigma}\text{conv}(H^1)\!=\!\left\{\!(x,y,z) \! \in \! [0,1] \!\times\R_+^{m+1} \!: \! {\left(\!\sqrt{\sigma}\!+\!(\sqrt{c\!+\!\sigma}- \!\sqrt{\sigma})x\right)^2
	\!\!+\! \!\sum_{i\in M} \! d_iy_i^2}\leq z^2\right\}.\end{equation} 
Moreover, from Proposition 8 of \cite{Modaresi2016} we find that $$\text{conv}(K)=\left\{(y,z)\in \R_+^{m+2}:{\left(\sqrt{\sigma}+\frac{\sqrt{\sigma\!+\!c}-\sqrt{\sigma}}{\sqrt{c}}y_0\right)^2+\sum_{i\in M}d_iy_i^2}\leq z^2\right\}.$$ Thus, the results coincide again. 

\subsubsection{Comparison with \ins{conic} MIR inequalities}
We now compare inequalities~\eqref{eq:polymatroidUnbounded} with the \emph{simple nonlinear conic mixed-integer rounding inequality} given in \cite{Atamturk2010}. Letting $a=\sqrt{\sigma}+\sqrt{\sigma+c}$ and $b=\frac{\sqrt{\sigma}}{a}$, we can write 
$$H^1=\left\{(x,y,z)\in \{0,1\}\times\R_+^{m+1}:{(x-b)^2+\sum_{i\in M} d_i\frac{y_i^2}{a^2}} \leq \frac{z^2}{a^2}\right\}.$$ Note that if $\sigma=0$ then $b=0$ and the MIR inequalities reduces to the original inequality--which defines the convex hull of $H^1$. If $\sigma>0$, then $\lfloor b \rfloor = 0$ and the simple mixed integer rounding inequality is 
\begin{align*}&{\left((1-2b)x+b\right)^2+\sum_{i\in M} d_i\frac{y_i^2}{a^2}}\leq \frac{z^2}{a^2} \\
\Leftrightarrow &{\left((1-2\frac{\sqrt{\sigma}}{\sqrt{\sigma}+\sqrt{\sigma+c}})x+\frac{\sqrt{\sigma}}{a}\right)^2+\sum_{i\in M} d_i\frac{y_i^2}{a^2}}\leq \frac{z^2}{a^2}\\
\Leftrightarrow &{\left((\frac{\sqrt{\sigma+c}-\sqrt{\sigma}}{a})x+\frac{\sqrt{\sigma}}{a}\right)^2+\sum_{i\in M} d_i\frac{y_i^2}{a^2}}\leq \frac{z^2}{a^2},
\end{align*}
and multiplying both sides by $a^2$ we get \eqref{eq:convexHullH1Sigma}.

\subsection{Set $R_X$ with rotated cone}
\label{sec:rotated}
Here we consider the set $R_X$ and, more generally, sets of the form written in conic quadratic form
\begin{equation*}
\rev{U_R=\left\{x\in X, (y,w,z)\in \R_+^{m+2}: {h(x)+\sum_{i\in M}d_iy_i^2+(w-z)^2}\leq (w+z)^2 \right\},}
\end{equation*}
where $h:X\to \R_+$.

Observe that the approach discussed in Section~\ref{sec:unbounded} can be used for $R_X$ and $U_R$. For example, using inequalities \eqref{eq:polymatroidUnbounded} for $R_X$ results in the valid inequalities
\begin{equation}
\label{eq:polymatroidRotated}
{\left(\sqrt{\sigma}+\pi'x\right)^2 +\sum_{i\in M}d_iy_i^2 +(w-z)^2}\leq (w+z)^2,\quad \pi\in \Pi_{\sigma}.
\end{equation}
We can also write inequalities \eqref{eq:polymatroidRotated} in rotated cone form, 
\begin{equation*}
(\sqrt{\sigma}+\pi'x)^2+\sum_{i\in M}d_iy_i^2 \leq 4wz,\quad \pi\in \Pi_\sigma.
\end{equation*} 

Note, however, that the second-order cone constraint defining $R_X$ and $U_R$ has additional structure, namely the continuous nonnegative variables $w$ and $z$ in both sides of the inequality. Nevertheless, as Proposition~\ref{prop:convexHullR} states, 
inequalities \eqref{eq:polymatroidRotated} are sufficient to characterize $\conv(R_X)$. The proof of Proposition~\ref{prop:convexHullR} is provided in Appendix~\ref{sec:appendix}.

\begin{proposition}
	\label{prop:convexHullR}
	The convex hull of $U_R$ is described as
	\small
	$$\rev{\text{conv}(U_R)=\left\{(x,y,w,z)\in \R_+^{n+m+2}:\exists s \text{ s.t. }(x,s)\in \conv(U_0)\text{ and } {s^2+\sum_{i\in M}d_iy_i^2}\leq 4wz\right\} \cdot}$$
\end{proposition}

\begin{remark} \label{remark:ratio}
	Consider again the set given by \eqref{eq:pOrderConstraint}--\eqref{eq:pOrderNonnegativity} in Section~\ref{sec:probabilityMax}, and observe that it corresponds to $U_R$ with $m=0$ and $\rev{U_0}=\left\{x\in X,\; s\in \R_+: 2\sqrt[\uproot{3}4]{c'x}\leq s\right\}$. Thus, if $X=\left\{0,1\right\}^n$, then $\text{conv}(U_R)$ is given by bounds constraints and inequalities
	$$
	(\rev{\xi}'x)^2\leq wz,\; \rev{\xi} \in \Pi(\bar{h}),
	$$
	where $\Pi(\bar{h})$ is the set of extreme points of the \rev{extended} polymatroid associated with the submodular function $\bar{h}(s)=2\sqrt[\uproot{3}4]{c'x}$. 
\end{remark}

\section{The case of bounded continuous variables}
\label{sec:bounded}
In this section we study $H_\rev\G$ with bounded continuous variables, i.e.,  by scaling $\rev\G=\left\{0,1\right\}^n\times [0,1]^m$. We first give a description of $\text{conv}(H_\rev\G)$ for the case $n=m=1$ and discuss the difficulties in obtaining the convex hull description for the general case (Section~\ref{sec:constraintsBounded}). Then we describe valid \rev{conic quadratic} inequalities that can be \rev{used with off-the-shelf solvers} (Section~\ref{sec:validBounded}).

\subsection{Two variable case with a bounded continuous variable}
\label{sec:constraintsBounded}
In this section we study the three-dimensional set
$${L}=\left\{(x,y,z)\in \{0,1\}\times [0,1] \times \R_+:\sqrt{\sigma+cx+dy^2}\leq z\right\},$$
where $\sigma\geq 0$ is a constant. First we give its convex hull description.

\begin{proposition}
	\label{prop:convexHullL}
	The convex hull of $L$ is described as
$$\text{conv}(L)= \left\{(x,y,z)\in [0,1]\times [0,1] \times \R_+:g(x,y)\leq z\right\},
\text{ where}$$ 
$$g(x,y) \! = \! \! \begin{cases}g_1(x,y) \!= \! \sqrt{\left(\sqrt{\sigma}+x(\sqrt{c+\sigma}-\sqrt{\sigma}) \right)^2+dy^2} \! \! \! & \text{if }y\leq x+(1-x)\sqrt{\frac{\sigma}{\sigma+c}}\\
g_2(x,y) \! = \! \sqrt{\sigma(1-x)^2+d(y-x)^2}+x\sqrt{\sigma+c+d} \! \! \! & \text{otherwise.}\end{cases}$$
\end{proposition}
\begin{proof}
	A point $(x,y,z)$ belongs to $\text{conv}(L)$ if and only if there exist $x_1,x_2,y_1,y_2,z_1,z_2,$ and $0\leq\lambda\leq 1$ such that the system
	\begin{align}
	x&=(1-\lambda)x_1+\lambda x_2 \label{eq:condX}\\
	y&=(1-\lambda)y_1 + \lambda y_2\label{eq:condY}\\
	z&=(1-\lambda)z_1 + \lambda z_2\label{eq:condZ}\\
	z_1&\geq \sqrt{\sigma+dy_1^2}\label{eq:cond1}\\
	z_2&\geq \sqrt{\sigma+c+d y_2^2}\label{eq:cond2}\\
	0&\leq y_1,y_2 \leq 1,\; x_1=0,\; x_2=1\label{eq:condBounds}
	\end{align}
	is feasible. Observe that from \eqref{eq:condX} and \eqref{eq:condBounds} we can conclude that $\lambda =x$.
	Also observe that from \eqref{eq:condX}, \eqref{eq:cond1} and \eqref{eq:cond2} we have that
	\begin{align*}z&=(1-x)z_1 + x z_2\\
	\Leftrightarrow z&\geq (1-x)\sqrt{\sigma+dy_1^2}+x\sqrt{\sigma+c+d y_2^2}.
	\end{align*}
	Therefore, the system is feasible if and only if
	\begin{align}
	z\geq \min_{y_1,y_2}\;& (1-x)\sqrt{\sigma+dy_1^2}+x\sqrt{\sigma+c+d y_2^2}\label{eq:objConvexHull}\\
	\text{s.t.}\; & y=(1-x)y_1 +x y_2 \tag{$\gamma$}\\
	\rev{(\text{CH})} \ \ \ \ \ \ \ \ \ \ \ \ \  \ \ & y_1\leq 1 \tag{$\alpha_1$}\\
	& y_2\leq 1 \tag{$\alpha_2$}\\
	& y_1\geq 0\tag{$\beta_1$}\\
	& y_2\geq 0, \tag{$\beta_{2}$}
	\end{align}
	and let $\gamma$, $\alpha\rev{=(\alpha_1,\alpha_2)}$ and $\beta\rev{=(\beta_1,\beta_2)}$ be the dual variables of the optimization problem above. Note that the objective function is differentiable even if $\sigma=0$ since in that case the function $\sqrt{\sigma+dy_1^2}$ reduces to the linear function $\sqrt{d}y_1$. Moreover, the optimization problem is convex, and from KKT conditions for variables $y_1$ and $y_2$ we find that
	\begin{align}
	-(1-x)\frac{d y_1}{\sqrt{\sigma +d y_1^2}}&=\gamma (1-x)+\alpha_{1}-\beta_{1}\notag\\
	-x\frac{d y_2}{\sqrt{\sigma+c+d y_2^2}}&=\gamma x+\alpha_{2}-\beta_{2}\notag\\
	\implies \frac{ y_1}{\sqrt{\sigma+dy_1^2}}+\bar{\alpha}_{1}-\bar{\beta}_{1}&=\frac{ y_2}{\sqrt{\sigma+c+d y_2^2}}+\bar{\alpha}_{2}-\bar{\beta}_{2}, \label{eq:KKTconditions}
	\end{align}
	where $\bar{\alpha}, \bar{\beta}$ correspond to $\alpha$ and $\beta$ after scaling. We deduce from \eqref{eq:KKTconditions} and complementary slackness that $y_1, y_2>0$ (unless $y=0$) and that $y_1\leq y_2$:  if $y_1=0$ and $y_2>0$ then $\bar{\alpha}_1=\bar{\beta}_2=0$,  and \eqref{eq:KKTconditions} reduces to $-\bar{\beta}_{1}=\nicefrac{ y_2}{\sqrt{\sigma+c+d y_2^2}}+\bar{\alpha}_{2}$, which has no solution since the right-hand-side is positive; letting $y_2=0$ and $y_1>0$ results in a similar contradiction; and if $0<y_2< y_1$ then $\bar{\beta}_1=\bar{\alpha}_2=\bar{\beta}_2=0$ and \eqref{eq:KKTconditions} reduces to $\nicefrac{y_1}{\sqrt{\sigma+dy_1^2}}+\bar{\alpha}_1=\nicefrac{ y_2}{\sqrt{\sigma+c+d y_2^2}}$, which has no solution since $y_1> y_2$ implies that $\nicefrac{y_1}{\sqrt{\sigma+dy_1^2}}>\nicefrac{y_2}{\sqrt{\sigma+dy_2^2}}>\nicefrac{y_2}{\sqrt{\sigma+c+dy_2^2}}$.
	
	Therefore, for an optimal solution either $0<y_1,y_2<1$ (and $\bar{\alpha}=\bar{\beta}=0$) or $y_2=1$ (and $\bar{\alpha}_2\geq 0$).
	If $\bar{\alpha}=\bar{\beta}=0$, then
	\begin{align*}
	y_1^*&= y\frac{\sqrt{\sigma}}{x\sqrt{c+\sigma}+(1-x)\sqrt{\sigma}}\quad \text{ and}\\
	y_2^*&= y\frac{\sqrt{c+\sigma}}{x\sqrt{c+\sigma}+(1-x)\sqrt{\sigma}}
	\end{align*}
	satisfy conditions \eqref{eq:condY} and \eqref{eq:KKTconditions}. \rev{Thus, if $y_2^*\leq 1$,} then $y_1^*, y_2^*$ also satisfy the bound constraints and correspond to an optimal solution to problem \rev{(CH)}. Replacing \rev{$(y_1,y_2)$ by their optimal values $(y_1^*,y_2^*)$} in \eqref{eq:objConvexHull}, we find that $$z\geq \sqrt{\left(\sqrt{\sigma}+x(\sqrt{c+\sigma}-\sqrt{\sigma}) \right)^2+dy^2}.$$
	\rev{The condition $y_2^*\leq 1$ is equivalent to 
		\begin{align*}
	y&\leq \frac{x\sqrt{c+\sigma}+(1-x)\sqrt{\sigma}}{\sqrt{c+\sigma}}=x+(1-x)\sqrt{\frac{\sigma}{c+\sigma}}.
		\end{align*}}
	
	 On the other hand, if $y_2^*>1$, an optimal solution to the optimization problem \rev{(CH)} is given by $\bar{y}_2=1$ and $\bar{y}_1=\frac{y-x}{1-x}$. Substituting \rev{$(y_1,y_2)$ by their optimal values} in \eqref{eq:objConvexHull},
	$$z\geq \sqrt{\sigma(1-x)^2+d(y-x)^2}+x\sqrt{\sigma+c+d}$$ when $y\geq x+(1-x)\sqrt{\frac{\sigma}{\sigma+c}}$.
\end{proof}

Note that inequality $g_1(x,y)\leq z$ is a special case of inequalities \eqref{eq:polymatroidUnbounded}. If $\sigma=0$, then we \rev{find that $g_2(x,y)\leq z$ reduces to $\sqrt{d}y+x(\sqrt{c+d}-\sqrt{d})\leq z$,}
which is a special case of inequalities \eqref{eq:polymatroidBoundedDominated}. However, inequality $g_2(x,y)\leq z$ is not valid if $\sigma>0$. In particular, it cuts off the feasible point $(x,y,z)=(1,0,\sqrt{\sigma+c})$. Moreover, it can be shown that the inequality $g_2(x,y)\leq z$ cuts off portions of $\text{conv}(L)$ whenever $y\leq x+(1-x)\frac{\sqrt{\sigma}}{\sqrt{\sigma+c}}$.

\begin{example}
Consider the set $L$ with $\sigma=d=1$ and $c=2$.
Figure~\ref{fig:example1} shows functions $g_1$ and $g_2$ when $x=0.5$ is fixed, and illustrates the comments above. We see that the function $g_2$ is always ``above" the function $g_1$, and cuts the convex hull of $L$ (the shaded region) whenever $y\leq x+(1-x)\frac{\sqrt{\sigma}}{\sqrt{\sigma+c}}$.
\begin{figure}[h]
\begin{center}
\includegraphics[width=0.7 \linewidth]{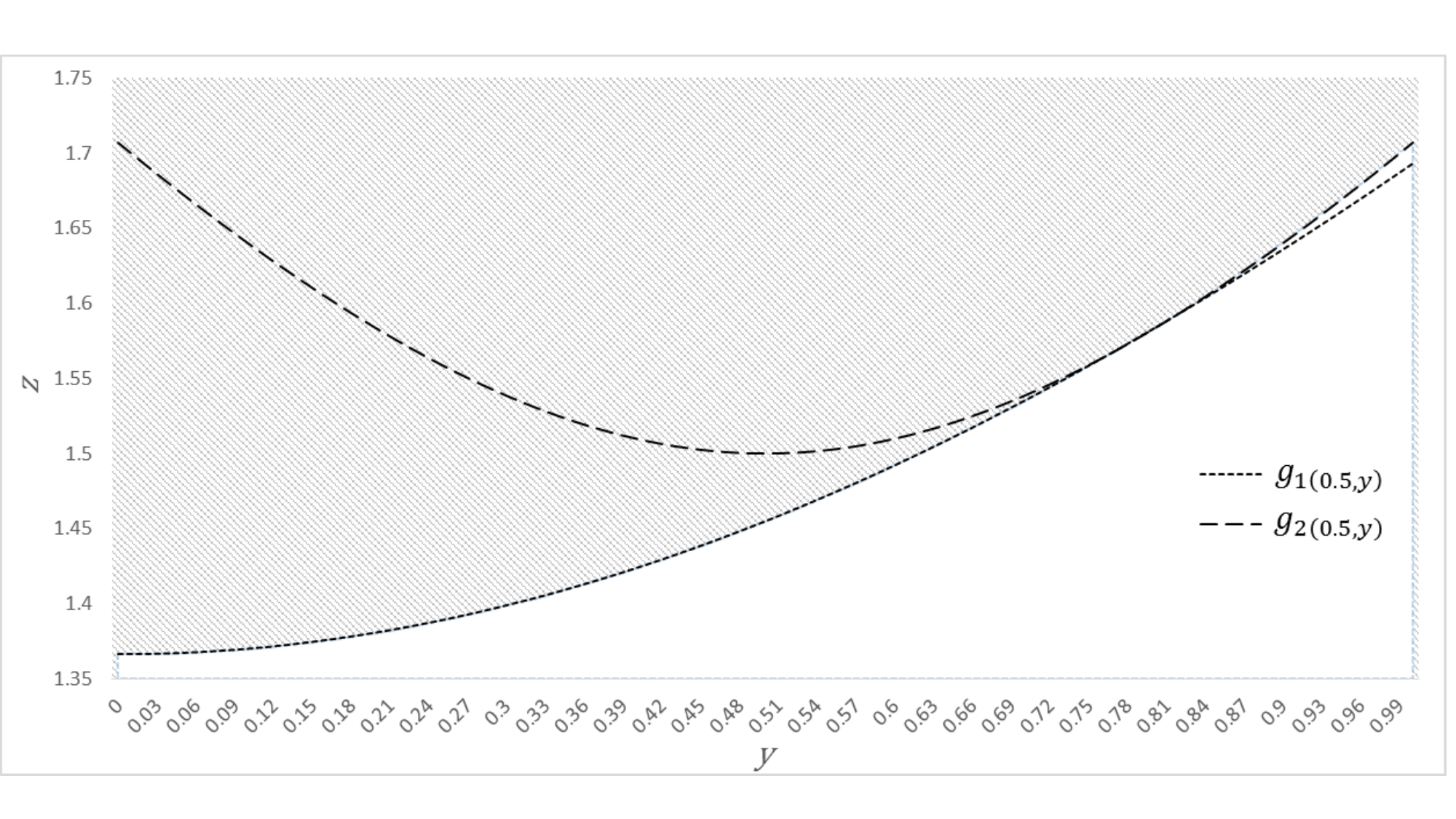}
\caption{Functions $g_1$, $g_2$ with $\sigma = d=1$, $c=2$ ($x=0.5$).}
\label{fig:example1}
\end{center}
\end{figure}
\end{example}

Unfortunately, Proposition~\ref{prop:convexHullL} does not help to describe the convex hull of $H_\rev\G$ with more than one bounded variable. Additionally, piecewise \rev{valid} functions like $g(x,y)$ in Proposition~\ref{prop:convexHullL} cannot be directly used with standard algorithms \rev{for convex mixed-integer optimization}. Thus, we now turn our attention to deriving inequalities that are valid and can be implemented as \rev{conic quadratic cuts}, if not sufficient to describe $\text{conv}(H_\rev\G)$ in general.

\subsection{The general (multi-variable) case}
\label{sec:validBounded}
To obtain valid inequalities for $H_\rev\G$ we write the conic quadratic constraint in extended form 
for a subset $T\subseteq M$ of the continuous variables:
\begin{align}
&{s^2+\sum_{i \in M\setminus T}d_iy_i^2}\leq z^2\notag\\
&\rev{\sigma+}\sum_{i\in N}c_i x_i + \sum_{i \in T}d_iy_i^2 \leq s^2. \label{eq:singleBounded}\\
&x\in \{0,1\}^n, y\in [0,1]^{M}, s \ge 0.  \notag
\end{align} 
Applying inequality \eqref{eq:polymatroidBoundedDominated} to \eqref{eq:singleBounded}
and eliminating the auxiliary variable $s$, we obtain the inequalities
\begin{equation}
\label{eq:polymatroidBounded}
{\left(\sqrt{\sigma+\sum_{i\in T}d_iy_i^2}+\pi'x\right)^2 +\sum_{i\in M\setminus T}d_iy_i^2}\leq z^2,\quad \pi\in \Pi_{\sigma+d(T)}.
\end{equation}
\begin{proposition}
	\label{prop:validityBounded}
	For $T\subseteq M$
	inequalities \eqref{eq:polymatroidBounded} are valid for $H_\rev\G$.
\end{proposition}

Note that inequalities \eqref{eq:polymatroidBounded} generalize or strengthen the previous valid inequalities proposed in this paper and other inequalities in the literature.
\begin{remark}
	\label{rem:dominateBounded}
	For $T=\emptyset$ inequalities \eqref{eq:polymatroidBounded} coincide with inequalities \eqref{eq:polymatroidUnbounded}. For $T=M$  inequalities \eqref{eq:polymatroidBounded} coincide with inequalities \eqref{eq:polymatroidBoundedDominated}. If $T\subset M$, then inequalities \eqref{eq:polymatroidBounded} dominate inequalities \eqref{eq:polymatroidBoundedDominated}.
\end{remark}

\setcounter{example}{0}
\begin{example}[Continued]
	We obtain from $\eqref{eq:polymatroidBounded}$ the valid inequality
	$$g_3(x,y)=\sqrt{\sigma+dy^2}+x\left(\sqrt{\sigma+c+d}-\sqrt{\sigma+d}\right)\leq z$$
	for $L$. 
	As Figure~\ref{fig:example2} shows, the inequality provides a good approximation of $L$ for the example considered.
	
	\begin{figure}[h]
		\begin{center}
			\includegraphics[width=0.7 \linewidth]{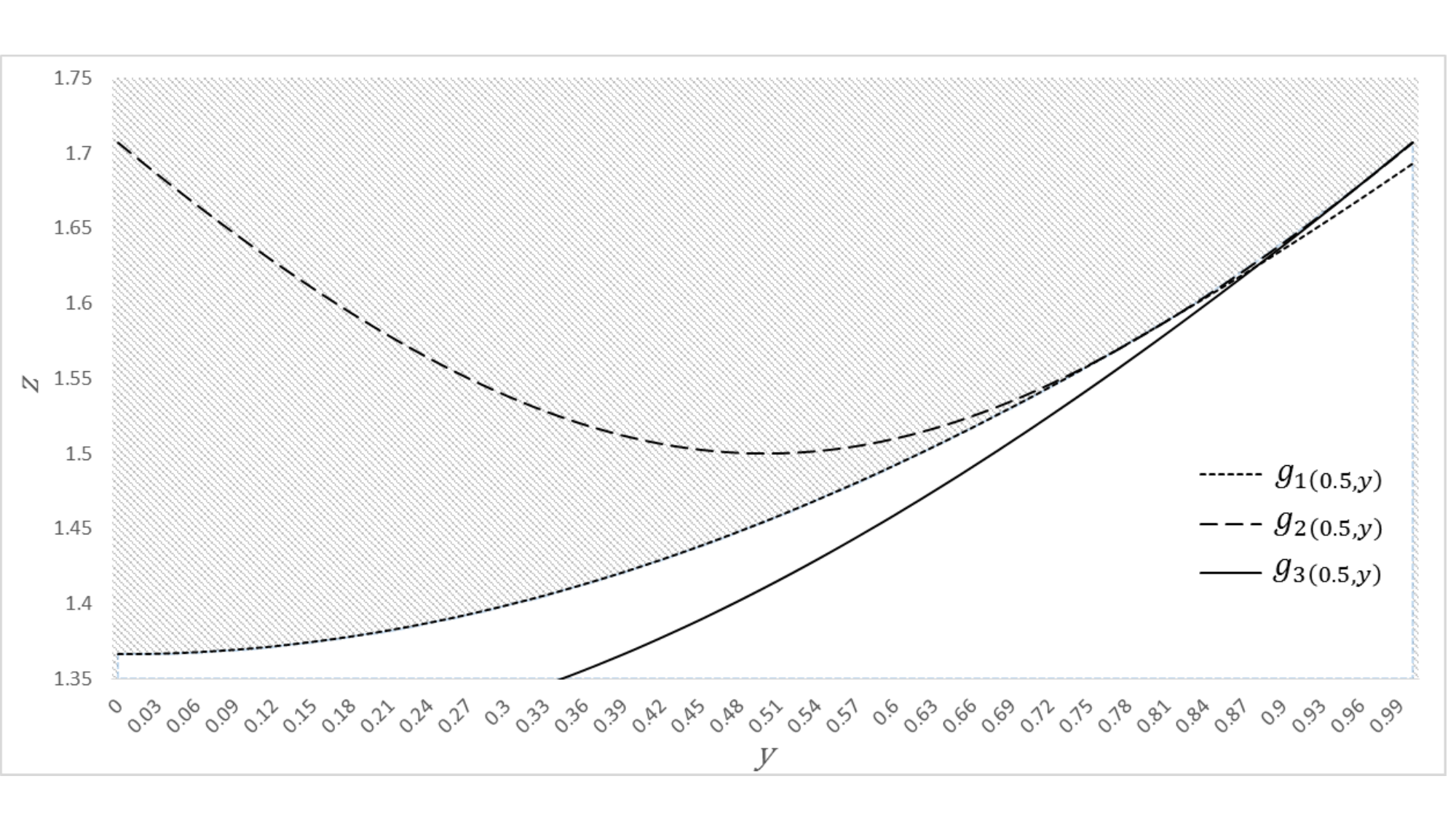}
			\caption{Functions $g_1$, $g_2$, $g_3$ with $\sigma=d=1$, $c=2$ $(x=0.5)$.}
			\label{fig:example2}
		\end{center}
	\end{figure}
\end{example}

\section{Valid inequalities for general $H_X$}
\label{sec:constrained}

In this section we derive  inequalities that exploit the structure for an arbitrary set $X \subseteq \rev\D$.
We first describe a lifting procedure for obtaining valid inequalities for \emph{any} mixed-binary set $X$, where computing each coefficient requires solving an integer optimization problem (Section~\ref{sec:validInequalitiesConstrained}). Then, we propose an approach based on linear programming to efficiently compute weaker valid inequalities (Section~\ref{sec:implementationConstrained}).

\subsection{General mixed-binary set $X$}
\label{sec:validInequalitiesConstrained}
We now consider valid inequalities for $H_X$ where $X\subseteq \rev\D$. The inequalities described here have a structure similar to the nonlinear extended polymatroid inequalities \eqref{eq:polymatroidUnbounded} and \eqref{eq:polymatroidBounded}.  
For a given a permutation $\left((1),(2),\ldots,(n)\right)$ of $N$ and $T\subseteq M$, let
\begin{align}
h_k(x,y)&=\sigma+\sum_{i=1}^{k-1}c_{(i)} x_{(i)}+\sum_{i\in T}d_iy_i^2\notag\\
\bar{\sigma}_{(k)}&=\max\left\{h_k(x,y) : (x,y)\in X, x_k=1\right\}, \text{ and}\label{eq:definitionSigma}\\
\rho_{(k)}&=\begin{cases}\sqrt{c_{(k)} +\bar{\sigma}_{(k)}}-\sqrt{\bar{\sigma}_{(k)}}& \text{if }\bar{\sigma}_{(k)}<\infty\\ 0 & \text{otherwise.}\label{eq:definitionRho} \end{cases}
\end{align}
Consider the inequality
\begin{equation}
\label{eq:polymatroidConstrained}
{\left(\sqrt{\sigma+\sum_{i\in T}d_iy_i^2}+\sum_{i=1}^n \rho_{(i)}x_{(i)}\right)^2+\sum_{i\in M\setminus T} d_iy_i^2}\leq z^2.
\end{equation}

\begin{proposition}
\label{prop:validityConstrained}
For $T\subseteq M$
inequalities \eqref{eq:polymatroidConstrained} are valid for $H_X$.
\end{proposition}
\begin{proof}
Let
\begin{equation*}
\label{eq:pOrderConeConstrained}
H_X(T)=\left\{(x,y)\in X, s\geq 0: \sqrt{\sigma+\sum_{i\in N}c_ix_i +\sum_{i\in T} d_i y_i^2}\leq s\right\},
\end{equation*}
and consider the extended formulation of $H_X$ given by
\begin{equation*}
\hat{H}_X=\left\{(x,y,s)\in H_X(T), z\geq 0: \sqrt{s^2 +\sum_{i\in M\setminus T} d_i y_i^2}\leq z\right\}.
\end{equation*}
To prove the validity of \eqref{eq:polymatroidConstrained} for $H_X$, it is sufficient to show that
\begin{equation}
\label{eq:polymatroidConstrainedT}
\sqrt{\sigma+\sum_{i\in T}d_iy_i^2}+\sum_{i=1}^n \rho_{(i)}x_{(i)}\leq s
\end{equation}
is valid for $H_X(T)$. In particular, we prove by induction that
\begin{equation}
\label{eq:inductionStatement}
\sqrt{\sigma+\sum_{i\in T}d_iy_i^2}+\sum_{i=1}^k \rho_{(i)}x_{(i)}\leq \sqrt{\sigma+\sum_{i=1}^k c_{(i)}x_{(i)} +\sum_{i\in T} d_i y_i^2}
\end{equation}
for all $(x,y)\in X$ and $k=0,\ldots,n$.

\paragraph{\textit{Base case: }$k=0$}Inequality \eqref{eq:inductionStatement} holds trivially.

\paragraph{\textit{Inductive step}}Let $(\bar{x},\bar{y})\in X$, and suppose inequality \eqref{eq:inductionStatement} holds for $k-1$. Observe that if $\bar{x}_{(k)}=0$ or $\rho_{(k)}=0$, then inequality \eqref{eq:inductionStatement} clearly holds for $k$. Therefore, assume that $\bar{x}_{(k)}=1$ and $\bar{\sigma}_{(k)}<\infty$. We have
\begin{align}
\sqrt{\sigma+\sum_{i=1}^{k} c_{(i)}\bar{x}_{(i)} +\sum_{i\in T} d_i \bar{y}_i^2}&=\sqrt{h_k(\bar{x},\bar{y})+c_{(k)}}\notag\\
&=\sqrt{h_k(\bar{x},\bar{y})}+\left(\sqrt{h_k(\bar{x},\bar{y})+c_{(k)}}-\sqrt{h_k(\bar{x},\bar{y})}\right)\notag\\
&\geq\sqrt{h_k(\bar{x},\bar{y})}+\left(\sqrt{\bar{\sigma}_{(k)}+c_{(k)}}-\sqrt{\bar{\sigma}_{(k)}}\right)\label{eq:concavity}\\
&\geq \sqrt{\sigma+\sum_{i\in T}d_i\bar{y}_i^2}+\sum_{i=1}^k \rho_{(i)}\bar{x}_{(i)},\label{eq:induction}
\end{align}
where \eqref{eq:concavity} follows from $\bar{\sigma}_{(k)}\geq h_k(\bar{x},\bar{y})$ (by definition of $\bar{\sigma}_{(k)}$) and from the concavity of the square root function, and \eqref{eq:induction} follows from $\sqrt{h_k(\bar{x},\bar{y})}\geq \sqrt{\sigma+\sum_{i\in T}d_i\bar{y}_i^2}+\sum_{i=1}^{k-1} \rho_{(i)}\bar{x}_{(i)}$ (induction hypothesis) and from the definition of $\rho_{(k)}$.
\end{proof}

\begin{remark}
\rev{Note that $\sigma_{(k-1)}=\sigma+\sum_{i=1}^{k-1}c_{(i)}$ and, if $T=\emptyset$, then $$\bar \sigma_{(k)}=\max_{\substack{x\in X\\ x_{(k)}=1}}\sigma+\sum_{i=1}^{k-1}c_{(i)}x_{i}.$$ In particular, if $T=\emptyset$, then $\bar\sigma_{(k)}\leq \sigma_{(k-1)}$ and $\rho_{(k)}\geq \pi_{(k)}$. Thus,}
\label{rem:specialCaseUnbounded}
for $T=\emptyset$ and $X=\rev\D$,
inequalities \eqref{eq:polymatroidConstrained} reduce to inequalities \eqref{eq:polymatroidUnbounded}; for $T=\emptyset$ and $X\subset \rev\D$, \del{then}
inequalities \eqref{eq:polymatroidConstrained} dominate inequalities \eqref{eq:polymatroidUnbounded}.
\end{remark}

\begin{remark}
\label{rem:specialCaseBounded}
For $X=\rev\G$, \del{then}
inequalities \eqref{eq:polymatroidConstrained} reduce to inequalities \eqref{eq:polymatroidBounded}.  
For $X\subset \rev\G$,
inequalities \eqref{eq:polymatroidConstrained} dominate inequalities \eqref{eq:polymatroidBounded}.
\end{remark}

\begin{remark}
For the case of the pure-binary set with defined by a cardinality constraint, i.e., $Y=\left\{x\in \{0,1\}^n:\sum_{i=1}^n x_i\leq k\right\}$ and $\sigma=0$, \cite{Yu2015} give facets for conv$(H_Y)$. However, noting the computation burden of constructing them, they propose approximate lifted inequalities of the form
$\sum_{i\leq k}\pi_{(i)}x_{(i)}+\sum_{i>k}\rho_{(i)}x_\rev{(i)}\leq z$,
where $\pi$ are computed according to \eqref{eq:definitionPi}, and $$\rho_{(i)}=\sqrt{c(T_{(i)})+c_\rev{(i)}}-\sqrt{c(T_{(i)})}$$ 
with $T_{(i)}={\arg\max}\left\{c(T): T\subseteq \{(1),\ldots,(i-1)\},\; |T|=k-1 \right\}.$
Thus, their approximate lifted inequalities coincide with inequalities \eqref{eq:polymatroidConstrained} and can be computed in $O(n\log n)$. If the set $X$ has additional constraints, then inequalities \eqref{eq:polymatroidConstrained} are stronger than the approximate lifted inequalities of \cite{Yu2015}.
\end{remark}

\begin{remark}
The strengthened extended polymatroid inequalities described in this section can be used with rotated cone constraints as well. In particular, for the set $$
R_X=\left\{(x,y)\in X, w\geq 0,z\geq 0:\sigma+ \sum_{i\in N}c_ix_i +\sum_{i\in M} d_i y_i^2\leq 4wz\right\},$$ we find that inequalities 
\begin{equation}
\label{eq:polymatroidRotatedConstrained}
\left(\sqrt{\sigma+\sum_{i\in T}d_iy_i^2}+\sum_{i=1}^n \rho_{(i)}x_{(i)}\right)^2+\sum_{i\in M\setminus T} d_iy_i^2\leq 4 wz
\end{equation} are valid for $R_X$.
\end{remark}

\ignore{

The optimization problem \eqref{eq:definitionSigma} may not be easy to solve exactly.
We propose in Section~\ref{sec:implementationConstrained} efficient ways to compute approximate coefficients.

Observe that the coefficients \eqref{eq:definitionRho} are obtained through a procedure similar to lifting. In particular, as Example \ref{ex:exampleKnapsack} shows, the coefficients \eqref{eq:definitionRho} can be interpreted as lower bounds of the lifting coefficients obtained when a subset of the discrete variables is assumed to be fixed at $0$.

\begin{example}[Lifting with knapsack constraint]
\label{ex:exampleKnapsack}
Given $a\in \R_+^n$ and $r>0$, consider the set
\begin{equation*}
G_C=\left\{(x,z)\in \{0,1\}^n\times \R_+: \sqrt[\uproot{20}p]{\sum_{i=1}^n c_ix_i}\leq z,\; \sum_{i=1}^n a_ix_i \leq r\right\}.
\end{equation*}
Let $G_{C,k}=\left\{(x,z)\in G_C: x_i=0,\; i=k,\ldots,n\right\}$ and suppose the inequality
\begin{equation}
\label{eq:partialInequality}
\sum_{i=1}^{k-1} \rho_i x_i\leq z
\end{equation}
is valid for $G_{C,k}$. Then inequality $\sum\limits_{i=1}^{k-1} \rho_i x_i +\rho_k x_k\leq z$ is valid for $G_{C,k+1}$ if and only if $\rho_k\leq \rho^*$, where
\begin{align}
\rho^*&= \min_{\substack{(x,z)\in G_{C,k+1} \\ x_{(k)}=1}}\left\{z-\sum_{i=1}^{k-1} \rho_{i}x_{i}\right\}\notag\\
&= \min_{x\in \{0,1\}^{k-1}}\left\{\sqrt[\uproot{20}p]{\sum_{i=1}^{k-1}c_{i}x_{i}+c_k}-\sum_{i=1}^{k-1} \rho_{i}x_{i}: \sum_{i=1}^{k-1}a_ix_i\leq r -a_k\right\}.\label{eq:computationExact}
\end{align}
Thus, computing the exact lifting coefficient requires solving a nonlinear discrete optimization problem. Now observe that from the validity of inequality \eqref{eq:partialInequality} we have that
\begin{equation}
\label{eq:partialValidity}
\sum_{i=1}^{k-1} \rho_{i}x_{i}\leq \sqrt[\uproot{20}p]{\sum_{i=1}^{k-1} c_ix_i},
\end{equation}
and we can use inequality \eqref{eq:partialValidity} to find a lower bound for $\rho^*$. In particular,
\begin{align}
\label{eq:lowerBound}
\rho^*&\geq \min_{x\in \{0,1\}^{k-1}}\left\{\sqrt[\uproot{20}p]{\sum_{i=1}^{k-1}c_{i}x_{i}+c_k}-\sqrt[\uproot{20}p]{\sum_{i=1}^{k-1} c_ix_i}: \sum_{i=1}^{k-1}a_ix_i\leq r -a_k\right\}.
\end{align}
By concavity of the root function, optimal solutions to \eqref{eq:lowerBound} correspond to optimal solutions to the optimization problem $$\max_{x\in \{0,1\}^{k-1}}\left\{ \sum_{i=1}^{k-1} c_ix_i: \sum_{i=1}^{k-1}a_ix_i\leq r -a_k\right\},$$
which corresponds to \eqref{eq:definitionSigma}.
 Thus, the coefficient $\rho_{(k)}$ given by \eqref{eq:definitionRho} is a lower bound on the exact lifting coefficient $\rho^*$ obtained by lifting from $0$. Moreover, note that if $a_i\in \Z_+^n$ and $r\in \Z_+$, then all the coefficients \eqref{eq:definitionRho} can be computed in $O(nr)$ using dynamic programming, while it is unclear how to solve the optimization problem \eqref{eq:computationExact}.
\end{example}

Finally, note that lifting requires some knowledge about the set $X$, to determine the variables assumed to be fixed. Moreover, in some cases, it is unclear how to use lifting at all (e.g., path constraints, since fixing variables often results in infeasible lifting problems). On the other hand, when computing coefficients \eqref{eq:definitionRho}, we do not assume that a subset of the variables is fixed. Thus, we can use inequalities \eqref{eq:polymatroidConstrained} for general sets $X$ without requiring any prior knowledge about the set.

}

\subsection{\rev{Relaxed inequalities}}
\label{sec:implementationConstrained}
Note that computing each coefficient of inequality \eqref{eq:polymatroidConstrained} requires solving a non-convex mixed 0-1 optimization problem \eqref{eq:definitionSigma}, which may not be practical in most cases. However, observe from Remarks~\ref{rem:specialCaseUnbounded} and \ref{rem:specialCaseBounded} that solving the optimization problem over \emph{any} relaxation of $X$ that includes the bound constraints results in valid inequalities at least as strong as the ones resulting from using only the bound constraints.

In particular, assume in problem \eqref{eq:definitionSigma} that, for $i\in T$, $y_i$ has a finite upper bound (otherwise the problem is unbounded and $\rho_i=0$) and $u_i=1$ (by scaling). Moreover let $X_P$ be a polyhedron such that $X\subseteq X_P$.
Convex constraints can also be included in $X_P$ by using a suitable linear outer approximation \citep{BenTal2001,Tawarmalani2005,Hijazi2013,Lubin2016}.

Given $X_P$, the approximate coefficients
\begin{align}
\hat{\rho}_{(k)}&=\sqrt{c_{(k)} +\hat{\sigma}_{(k)}}-\sqrt{\hat{\sigma}_{(k)}}\text{, with}\label{eq:definitionRhoHat}\\
\hat{\sigma}_{(k)}&=\sigma+\max\left\{\sum_{i=1}^{k-1}c_{(i)} x_{(i)}+\sum_{i\in T}d_iy_i : (x,y)\in X_P, x_k=1\right\}\notag
\end{align}
can be computed efficiently by solving $n$ linear programs. Moreover, the linear program required to compute $\hat{\sigma}_{(k)}$ differs from the one required for $\hat{\sigma}_{(k-1)}$ in two bound constraints, corresponding to $x_{(k-1)}$ and $x_{(k)}$, and one objective coefficient, corresponding to  $x_{(k-1)}$. Therefore, using the simplex method with warm starts, each $\hat{\sigma}_{(k)}$ can be computed efficiently, using only a small number of simplex pivots.

\rev{
\section{Computational considerations}
\label{sec:compcons}
Table~\ref{tab:taxonomy} presents a classification of the proposed inequalities, depending on whether the continuous variables are bounded or not, on whether the inequalities are for the set with the \rep{conic quadratic}{Lorentz} cone $H_X$ or the rotated cone $R_X$, and on whether additional constraints are used to strengthen the inequalities (\texttt{strengthened}) or not (\texttt{polymatroid}). Note that there is a direct correspondence between the inequalities for \rep{conic quadratic}{Lorentz} cones and for rotated cones and, although not explicitly shown in the paper, it is easy to construct the rotated cone version of inequality \eqref{eq:polymatroidBounded}.
\begin{table}[!h]
	\caption{Classification of the proposed inequalities.}
	\label{tab:taxonomy}
	\begin{center}
		\begin{tabular}{c c c c c}
			\hline
			\multirow{2}{*}{\texttt{Continuous variables}} & \multicolumn{2}{c}{\underline{\texttt{polymatroid}}} &  \multicolumn{2}{c}{\underline{\texttt{strengthened}}}\\
			& \texttt{conic quad} & \texttt{rotated} & \texttt{conic quad} & \texttt{rotated}\\
			\hline
			Unbounded & \eqref{eq:polymatroidUnbounded} & \eqref{eq:polymatroidRotated} &\eqref{eq:polymatroidConstrained}, $T=\emptyset$ & \eqref{eq:polymatroidRotatedConstrained}, $T=\emptyset$\\
			Bounded &\eqref{eq:polymatroidBounded} & & \eqref{eq:polymatroidConstrained} & \eqref{eq:polymatroidRotatedConstrained}\\
			\hline
		\end{tabular}
	\end{center}
\end{table}

We now consider the implementation of the proposed inequalities in branch-and-cut algorithms. First, in Section~\ref{sec:implementationBounded}, we discuss the difficulties in using the inequalities for the (more general) bounded case, then in Section~\ref{sec:implementationUnbounded} we show how to efficiently use the cuts for the unbounded case.

\subsection{Bounded case}
\label{sec:implementationBounded}

For brevity, we only discuss inequalities \eqref{eq:polymatroidBounded} of the form $\varphi(x,y)\leq z$, where
$$\varphi(x,y)=\sqrt{\left(\sqrt{\sigma+\sum_{i\in T}d_iy_i^2}+\pi'x\right)^2+\sum_{i\in M\setminus T}d_iy_i^2}.$$ All \del{the} other inequalities for the bounded case have \ins{a} similar structure, so the discussion extends directly to those cases as well. Inequalities \eqref{eq:polymatroidBounded} are nonlinear, and can be added to the formulation as nonlinear inequalities, or can be implemented via linear cutting planes using outer approximations. Unfortunately, both approaches have drawbacks which may limit the effectiveness of the inequalities in practice when used with current off-the-shelf solvers.

\subsubsection{Implementation as nonlinear cuts}
\label{sec:implementationNonlinearCuts}
The function $\varphi$ is \rep{conic quadratic}{SOCP}-representable; in particular, the inequality $\varphi(x,y)\leq z$ is equivalent to the system 
\begin{align}
 s_1^2&\geq\sigma+\sum_{i\in T}d_iy_i^2\label{eq:conic1}\\
s_2&=s_1+\pi'x\notag\\
z&\geq s_2^2+\sum_{i\in M\setminus T}d_iy_i^2\label{eq:conic2}\\
0&\leq s_1,s_2,\notag
\end{align}
where \eqref{eq:conic1} and \eqref{eq:conic2} are \del{second-order} conic \ins{quadratic} inequalities accepted by most solvers.

Observe that adding each inequality \eqref{eq:polymatroidBounded} requires two additional variables and conic constraints, thus adding even a modest number of inequalities may substantially increase the difficulty of solving the convex relaxations at each node of the branch-and-bound tree. Additionally, solvers rely on the dual simplex method to solve the subproblems arising in branch-and-bound (by constructing a linear approximation of non-polyhedral sets) due to its warm starts capabilities; adding nonlinear cuts such as \eqref{eq:conic1} and \eqref{eq:conic2} may render the existing simplex tableau ineffective and require solving the subproblems from scratch.
 Finally, commercial solvers, currently, do not allow adding nonlinear cuts during branching, and inequalities \eqref{eq:polymatroidBounded} need to be added \rep{by the user at the root node explicitly}{manually}, \rep{giving up}{thus} the benefits of built-in cut-management strategies. 

\subsubsection{Implementation as linear outer approximations}
\label{sec:implementationGradient}
Cutting planes based on a linear outer approximation of \ins{the convex function} $\varphi$ can be added using gradients. Given a fractional solution $(\bar{x},\bar{y})$, the linear underestimator $\bar \varphi(x,y)\leq z$, where $$\bar \varphi(x,y)=\varphi(\bar{x},\bar{y})+\nabla_x \varphi(\bar{x})'(x-\bar{x})+\nabla_y \varphi(\bar{y})'(y-\bar{y})$$
is valid. In particular, we find
\begin{equation*}
\bar \varphi(x,y)=\psi+\frac{1}{\psi}\left(\eta\pi'(x-\bar{x})+\zeta\sum_{i\in T}d_i\bar{y}_i (y_i-\bar{y}_i)+\sum_{i\in M\setminus T}d_i\bar{y}_i(y_i-\bar{y}_i)\right),
\end{equation*}
where
\begin{align*}
\eta=\sqrt{\sigma+\sum_{i\in T}d_i\bar{y}_i^2}+\pi'\bar{x}; \ \ 
\zeta=\frac{\eta}{\sqrt{\sigma+\sum_{i\in T}d_i\bar{y}_i^2}}; \ \
\psi=\sqrt{\eta^2 +\sum_{i\in M\setminus T}d_i\bar{y}_i^2}.
\end{align*}
An implementation based on the linear cuts $\bar\varphi(x,y)\leq z$ leverages the existing capabilities of current commercial solvers, including warm starts and cut management strategies. Nevertheless, each linear inequality $\bar\varphi(x,y)\leq z$ is often weak, and constructing a suitable approximation of the original nonlinear inequality $\varphi(x,y)\leq z$ may require a prohibitive number of cuts. 

In Appendix~\ref{sec:computationalBounded} we provide a comparison of both approaches for a simple mean-risk minimization problem with bounded continuous variables and no correlations. Adding the nonlinear inequalities directly, as discussed in Section~\ref{sec:implementationNonlinearCuts}, results in significantly better performance, both in terms of the relaxation quality and the solution times. These results are consistent with the recent experience by the authors using other classes of nonlinear inequalities, see \cite{atamturk2018strong} and \cite{gomez2018strong}.

\subsection{Unbounded case}
\label{sec:implementationUnbounded}
In most of the applications discussed in Section~\ref{sec:applications}, the continuous variables are used to model covariance terms, rotated cone constraints or denominators in fractional optimization. In such cases, the continuous variables are unbounded, and the proposed inequalities can \del{actually} be implemented efficiently in such settings. Observe that the conic quadratic inequality arising in set $H_X$ can be written in an extended formulation as 
\begin{align*}
s^2&\geq \sum_{i\in N} c_ix_i^2\\
z^2&\geq s^2+\sum_{i\in M}d_iy_i^2\\
0&\leq s.
\end{align*}
Similarly, the rotated cone inequality arising in set $R_X$ can be written as 
\begin{align*}
s^2&\geq \sum_{i\in N} c_ix_i^2 \\
t^2&\geq s^2+\sum_{i\in M}d_iy_i^2 \\
t^2&\leq wz\\
0&\leq s,t.
\end{align*}
In both cases, the \texttt{polymatroid} and \texttt{strengthened} inequalities can be added as linear cuts, $\pi'x\leq s$ and $\rho'x\leq s$, respectively. Thus, \rep{when}{by} adding the nonlinear inequalities as linear cuts in an extended formulation, \rep{optimization}{the resulting} algorithms benefit from the warm starts and cut management strategies without sacrificing \rep{the strength of the inequalities}{any strength}. Such \ins{a} formulation cannot be \rep{used effectively}{replicated} for the bounded case, since an additional variable would \rep{be needed}{need to be added} for \rep{each subset $T$}{every subset} of $M$.
}

\section{Experiments}
\label{sec:computational}

In this section we report computational experiments performed to test the effectiveness of the polymatroid inequalities in solving second order cone optimization with a branch-and-cut algorithm.  
In Section~\ref{sec:computationalNondiagonal} we solve instances with general covariance matrices (see application in Section \ref{sec:correlated}), in Section~\ref{sec:robustComputations} we solve \rev{conic quadratic interdiction problems} (see application in Section \ref{sec:robust}), and in Section~\ref{sec:fractionaComputations} we solve binary linear fractional problems (see applications in Section~\ref{sec:fractional}).

All experiments are done using CPLEX 12.6.2 solver on a workstation with a 2.93GHz Intel\textregistered Core\textsuperscript{TM} i7 CPU and 8 GB main memory and with a single thread. \rev{We compare using default CPLEX without adding any cuts (\texttt{cpx}), using the inequalities \del{proposed} in Section~\ref{sec:unbounded} (\texttt{polymatroid}) and using the strengthened inequalities \del{proposed} in Section~\ref{sec:constrained} (\texttt{strengthened}). Since in all cases the continuous variables are unbounded, we implement the inequalities as discussed in Section~\ref{sec:implementationUnbounded}.} The time limit is set to two hours and CPLEX' default settings are used.
The inequalities are added only at the root node using callback functions, and all times reported include the time required to add cuts.  
\ignore{
\begin{description}
\item[None] No cutting planes are added.
\item[Poly1] Inequalities \eqref{eq:polymatroidUnbounded} are added as cutting planes. We use an extended formulation and add the inequalities as linear cuts as described in Remark~\ref{rem:extended}.
\item[Poly2] Inequalities \eqref{eq:polymatroidConstrained} are added as cutting planes. The specific implementation depends on the instance.
\end{description}
}

\ignore{

\subsection{Instances with a cardinality constraint}
\label{sec:computationalCardinality}

In this section we test the value of strengthening the polymatroid inequalities utilizing additional problem constraints. To do so, we solve optimization problems with a cardinality constraint: \begin{equation}
\label{eq:computationalB}
\min_{x\in \{0,1\}^n}\left\{-a'x+\Omega \sqrt{c'x} : \sum_{i= 1}^n x_i\leq k\right\},
\end{equation}
where $a$ and $c$ are generated as in Section~\ref{sec:computationalBounded} and $\Omega=\Phi^{-1}(\alpha)$, where $\Phi$ is the cumulative distribution function of the normal distribution and $\alpha\in \{0.95,0.975,0.99\}$. We set $n=200$, and set $k$ to be 15\%, 20\% and 25\% of the total number of variables.
Inequalities \eqref{eq:extendedPolymatroidInequality} and \eqref{eq:polymatroidConstrained} are compared with default CPLEX. The coefficients of inequalities \eqref{eq:polymatroidConstrained} are computed using linear programming with warm starts as outlined in Section~\ref{sec:implementationConstrained}---observe that, in this case, the coefficients \eqref{eq:definitionRhoHat} coincide with \eqref{eq:definitionRho} since the feasible region is an integral polytope.

 \ignore{ 
 the initial gap, and for each type of valid inequalities: the initial gap, the root improvement, the number of branch-and-bound nodes, the time used in seconds (T), the end gap (E) and number of instances solved to optimality([\#]). Each row represents the average over five instances generated with the same parameters.
}

{
\renewcommand\arraystretch{0.8}
\begin{table}[h!]
\caption{Experiments with cardinality constraints.}
\setlength{\tabcolsep}{2pt} 
\begin{center}
\label{tab:card0}
\SingleSpacedXI
\scalebox{0.7}{
\begin{tabular}{ c c c |c  c c c| c c c c| c c c c}
\hline \hline
\multirow{2}{*}{$k$} & \multirow{2}{*}{$\alpha$} & \multirow{2}{*}{\textbf{\texttt{igap}}} & \multicolumn{4}{c|}{\textbf{\texttt{cpx}}} & \multicolumn{4}{c|}{\textbf{\texttt{inequality \eqref{eq:extendedPolymatroidInequality}}}}& \multicolumn{4}{c}{\textbf{\texttt{inequality \eqref{eq:polymatroidConstrained}}}}\\
&&&\texttt{rimp}&\texttt{nodes}&\texttt{time}&\texttt{egap}[\#]&\texttt{rimp}&\texttt{nodes}&\texttt{time}&\texttt{egap}[\#]&\texttt{rimp}&\texttt{nodes}&\texttt{time}&\texttt{egap}[\#]\\
\midrule
\multirow{3}{*}{30}& 0.95 & 4.4 & 23.7 & 7,150,715 & 2,528 & 0.3[4] & 36.6 & 3,754,826 & 2,073 & 0.4[4] & 48.9 & 2,614,446 & 1,510 & 0.2[4]\\
& 0.975 & 7.2 & 7.2 & 13,632,197 & 6,120 & 1.8[1] & 23.7 & 9,573,199 & 5,945 & 1.8[1] & 39.8 & 9,235,158 & 5,797 & 1.0[1]\\
& 0.99 & 11.9 & 4.0 & 16,867,459 & 7,200 & 5.0[0] & 14.7 & 10,899,169 & 7,200 & 5.7[0] & 31.6 & 13,328,370 & 7,200 & 4.1[0]\\
\multicolumn{3}{c|}{\textbf{Average}}&\textbf{11.6} &\textbf{ 12,550,124} & \textbf{5,283} & \textbf{2.4[5]} & \textbf{25.0} & \textbf{8,075,731} & \textbf{5,073} & \textbf{2.6[5]} & \textbf{40.1} & \textbf{8,392,658} & \textbf{4,836} & \textbf{1.8[5]}\\
\midrule
&&&&&&&&&&
\\
\multirow{3}{*}{40}& 0.95 & 1.9 & 20.7 & 6,235,270 & 1,674 & 0.1[4] & 70.5 & 620,389 & 261 & 0.0[5] & 75.0 & 90,179 & 62 & 0.0[5]\\
& 0.975 & 3.3 & 9.6 & 13,961,488 & 4,360 & 0.4[3] & 49.2 & 3,268,824 & 2,122 & 0.2[4] & 57.3 & 2,729,459 & 1,557 & 0.2[4]\\
& 0.99 & 5.6 & 6.0 & 15,334,782 & 6,738 & 1.8[1] & 30.0 & 6,110,571 & 6,149 & 1.7[1] & 42.6 & 5,222,829 & 5,799 & 1.2[1]\\
\multicolumn{3}{c|}{\textbf{Average}}&\textbf{12.1} &\textbf{ 11,843,847} & \textbf{4,257} & \textbf{0.8[8]} & \textbf{49.9} & \textbf{3,333,261} & \textbf{2,844} & \textbf{0.6[10]} & \textbf{58.3} & \textbf{2,680,821} & \textbf{2,472} & \textbf{0.5[10]}\\
\midrule
&&&&&&&&&&
\\
\multirow{3}{*}{50}& 0.95 & 1.0 & 8.9 & 270,852 & 72 & 0.0[5] & 93.3 & 249 & 2 & 0.0[5] & 93.3 & 98 & 2 & 0.0[5]\\
& 0.975 & 1.6 & 8.0 & 3,882,494 & 1,045 & 0.0[5] & 81.3 & 316,625 & 221 & 0.0[5] & 84.4 & 198,916 & 92 & 0.0[5]\\
& 0.99 & 2.8 & 7.9 & 14,835,539 & 4,600 & 0.3[3] & 57.3 & 4,695,268 & 3,480 & 0.2[3] & 64.3 & 983,894 & 1,537 & 0.2[4]\\
\multicolumn{3}{c|}{\textbf{Average}}&\textbf{8.3} &\textbf{ 6,329,628} & \textbf{1,906} & \textbf{0.1[13]} & \textbf{77.3} & \textbf{1,670,714} & \textbf{1,234} & \textbf{0.1[13]} & \textbf{80.7} & \textbf{394,293} & \textbf{544} & \textbf{0.1[14]}\\

\hline\hline
\end{tabular}
}
\end{center}
\end{table}
} 

Table~\ref{tab:card0} presents the results for each value of $k$ and $\alpha$.
We see that for instances with $k=50$, using inequalities \eqref{eq:extendedPolymatroidInequality} or \eqref{eq:polymatroidConstrained} results in gap improvement of more than 75\% and faster solutions times than default CPLEX. In particular, using inequalities \eqref{eq:polymatroidConstrained} results in solutions times that are four times faster than default CPLEX on average.
As expected, for instances with tighter cardinality constraints,  inequalities \eqref{eq:polymatroidConstrained}, which exploit the cardinality constraints, are more effective than inequalities \eqref{eq:extendedPolymatroidInequality} in reducing the solution times as well as end gaps. On the other hand, when the cardinality constraint is loose, the effectiveness of both classes of inequalities improve.

\ignore{
 are less effective. On the other hand using inequalities \eqref{eq:polymatroidConstrained}, which exploits the cardinality constraint, results in an improvement over default CPLEX, with better solution times and end gaps in all cases. Observe that the additional root gap improvement achieved by inequalities \eqref{eq:polymatroidConstrained} with respect to using inequalities \eqref{eq:extendedPolymatroidInequality} is greater in instances with tight cardinality constraints. Indeed, in such instances, the stronger coefficients \eqref{eq:definitionRho} are a larger improvement over coefficients \eqref{eq:definitionPi}.
}

}

\subsection{Mean-risk minimization with correlated random variables}
\label{sec:computationalNondiagonal}

In this section we test the effectiveness \rev{of} the polymatroid inequalities in instances with correlated random variables.
In particular, we solve mean-risk \rev{minimization} problems
\begin{equation}
\label{eq:computationalC}
\min_{x\in \{0,1\}^n}\left\{-a'x +\Omega \sqrt{x'Qx}: \sum_{i=1}^nx_i\leq k\right\},
\end{equation}
where the matrix $Q$ is generated according to a factor model, i.e., $Q=\rev ZF\rev Z'+D$ where $F\in \R^{r\times r}$ is the factor covariance matrix, $\rev Z\in \R^{n\times r}$ is the exposure matrix and $D\in \R^{n\times n}$ is diagonal matrix with the specific covariances. Observe that in such instances, we can set $\diag(c)=D$ in equation \eqref{eq:socpConstraint}. 

 In our experiments $F=GG'$, with $G\in \R^{r\times r}$ and $G_{ij}\sim U[-1,1]$, $\rev Z_{ij}\sim U[0,1]$ with probability $0.2$ and $\rev Z_{ij}=0$ otherwise, $D_{ii}\sim U[0,\delta \bar{q}]$, where $\delta\geq 0$ is a diagonal dominance parameter and $\bar{q}=\frac{1}{N}\sum_{i\in N}{Q_0}_{ii}$, and $a_i\sim U[0.85\sqrt{Q_{ii}},1.15\sqrt{Q_{ii}}]$. \rev{We set} the parameter $\Omega=\Phi^{-1}(\alpha)$, where $\Phi$ is the cumulative distribution function of the normal distribution \rev{and $\alpha\in\{0.95,0.975,0.99\}$}. 
  We let $n=200$, $r=40$ and $k$ equal to 10\%, 15\%, and 20\% of the number of the  variables.


{
\renewcommand\arraystretch{0.7}
\begin{table}[h!]
\caption{Experiments with general covariance matrices ($\delta=0.5$).}
\setlength{\tabcolsep}{2pt}
\begin{center}
\label{tab:card05delta}
\SingleSpacedXI
\scalebox{0.7}{
\begin{tabular}{ c c c |c  c c c| c c c c| c c c c}
\hline \hline
\multirow{2}{*}{$k$} & \multirow{2}{*}{$\alpha$} & \multirow{2}{*}{\textbf{\texttt{igap}}} & \multicolumn{4}{c|}{\textbf{\texttt{cpx}}} & \multicolumn{4}{c|}{\textbf{\texttt{\rev{polymatroid}}}}& \multicolumn{4}{c}{\textbf{\texttt{\rev{strengthened}}}}\\
&&&\texttt{rimp}&\texttt{nodes}&\texttt{time}&\texttt{egap}[\#]&\texttt{rimp}&\texttt{nodes}&\texttt{time}&\texttt{egap}[\#]&\texttt{rimp}&\texttt{nodes}&\texttt{time}&\texttt{egap}[\#]\\
\hline
\multirow{3}{*}{20}& 0.95 & 1.7 & 22.6 & 9,557 & 74 & 0.0[5] & 53.3 & 3,957 & 23 & 0.0[5] & 55.6 & 2,367 & 17 & 0.0[5]\\
& 0.975 & 3.0 & 21.3 & 33,468 & 242 & 0.0[5] & 53.5 & 13,316 & 86 & 0.0[5] & 55.9 & 5,839 & 40 & 0.0[5]\\
& 0.99 & 5.2 & 15.2 & 164,568 & 1,845 & 0.0[5] & 52.8 & 80,735 & 730 & 0.0[5] & 55.3 & 23,577 & 269 & 0.0[5]\\
\multicolumn{3}{c|}{\textbf{Average}}&\textbf{19.7} &\textbf{ 69,198} & \textbf{720} & \textbf{0.0[15]} & \textbf{53.2} & \textbf{32,669} & \textbf{280} & \textbf{0.0[15]} & \textbf{55.6} & \textbf{10,594} & \textbf{109} & \textbf{0.0[15]}\\
\hline
&&&&&&&&&&
\\
\multirow{3}{*}{30}& 0.95 & 0.8 & 15.5 & 7,115 & 57 & 0.0[5] & 53.3 & 1,656 & 11 & 0.0[5] & 52.4 & 1,159 & 9 & 0.0[5]\\
& 0.975 & 1.3 & 14.9 & 18,901 & 135 & 0.0[5] & 53.1 & 2,800 & 20 & 0.0[5] & 54.0 & 2,095 & 15 & 0.0[5]\\
& 0.99 & 2.3 & 5.7 & 76,675 & 1,005 & 0.0[5] & 61.1 & 8,265 & 48 & 0.0[5] & 62.1 & 5,131 & 30 & 0.0[5]\\
\multicolumn{3}{c|}{\textbf{Average}}&\textbf{12.0} &\textbf{ 34,230} & \textbf{399} & \textbf{0.0[15]} & \textbf{55.8} & \textbf{4,240} & \textbf{26} & \textbf{0.0[15]} & \textbf{56.2} & \textbf{2,795} & \textbf{18} & \textbf{0.0[15]}\\
\hline
&&&&&&&&&&
\\
\multirow{3}{*}{40}& 0.95 & 0.4 & 23.3 & 2,910 & 18 & 0.0[5] & 48.5 & 611 & 6 & 0.0[5] & 50.5 & 577 & 6 & 0.0[5]\\
& 0.975 & 0.7 & 20.0 & 4,216 & 30 & 0.0[5] & 54.3 & 884 & 7 & 0.0[5] & 55.5 & 839 & 7 & 0.0[5]\\
& 0.99 & 1.1 & 13.5 & 46,030 & 514 & 0.0[5] & 55.9 & 2,493 & 18 & 0.0[5] & 56.7 & 2,144 & 14 & 0.0[5]\\
\multicolumn{3}{c|}{\textbf{Average}}&\textbf{18.9} &\textbf{ 17,719} & \textbf{187} & \textbf{0.0[15]} & \textbf{52.9} & \textbf{1,329} & \textbf{10} & \textbf{0.0[15]} & \textbf{54.2} & \textbf{1,187} & \textbf{9} & \textbf{0.0[15]}\\
\hline\hline
\end{tabular}
}
\end{center}
\end{table}
} 
{
\renewcommand\arraystretch{0.7}
\begin{table}[h!]
\caption{Experiments with general covariance matrices ($\delta=1.0$).}
\setlength{\tabcolsep}{2pt}
\begin{center}
\label{tab:card10delta}
\SingleSpacedXI
\scalebox{0.7}{
\begin{tabular}{ c c c |c  c c c| c c c c| c c c c}
\hline \hline
\multirow{2}{*}{$k$} & \multirow{2}{*}{$\alpha$} & \multirow{2}{*}{\textbf{\texttt{igap}}} & \multicolumn{4}{c|}{\textbf{\texttt{cpx}}} & \multicolumn{4}{c|}{\textbf{\texttt{\rev{polymatroid}}}}& \multicolumn{4}{c}{\textbf{\texttt{\rev{strengthened}}}}\\
&&&\texttt{rimp}&\texttt{nodes}&\texttt{time}&\texttt{egap}[\#]&\texttt{rimp}&\texttt{nodes}&\texttt{time}&\texttt{egap}[\#]&\texttt{rimp}&\texttt{nodes}&\texttt{time}&\texttt{egap}[\#]\\
\hline
\multirow{3}{*}{20}& 0.95 & 2.9 & 21.6 & 64,283 & 927 & 0.0[5] & 55.1 & 14,984 & 165 & 0.0[5] & 59.1 & 6,233 & 68 & 0.0[5]\\
& 0.975 & 5.0 & 15.5 & 240,224 & 3,975 & 0.4[3] & 44.4 & 189,826 & 3,390 & 0.4[3] & 50.9 & 102,053 & 1,915 & 0.1[4]\\
& 0.99 & 9.0 & 6.4 & 378,116 & 7,200 & 2.2[0] & 35.7 & 477,553 & 7,200 & 1.9[0] & 43.1 & 430,707 & 5,966 & 0.6[2]\\
\multicolumn{3}{c|}{\textbf{Average}}&\textbf{14.5} &\textbf{ 227,541} & \textbf{4,034} & \textbf{0.9[8]} & \textbf{45.1} & \textbf{227,454} & \textbf{3,585} & \textbf{0.8[8]} & \textbf{51.0} & \textbf{179,664} & \textbf{2,650} & \textbf{0.2[11]}\\
\hline
&&&&&&&&&&
\\
\multirow{3}{*}{30}& 0.95 & 1.1 & 17.1 & 32,629 & 316 & 0.0[5] & 77.2 & 1,082 & 12 & 0.0[5] & 78.2 & 682 & 10 & 0.0[5]\\
& 0.975 & 2.0 & 12.5 & 150,756 & 2,046 & 0.1[4] & 72.9 & 12,202 & 107 & 0.0[5] & 75.5 & 4,896 & 39 & 0.0[5]\\
& 0.99 & 3.5 & 10.5 & 258,866 & 3,679 & 0.5[3] & 67.8 & 115,507 & 1,510 & 0.1[4] & 70.6 & 59,106 & 511 & 0.0[5]\\
\multicolumn{3}{c|}{\textbf{Average}}&\textbf{13.4} &\textbf{ 147,417} & \textbf{2,014} & \textbf{0.2[12]} & \textbf{72.6} & \textbf{42,930} & \textbf{543} & \textbf{0.0[14]} & \textbf{74.8} & \textbf{21,561} & \textbf{187} & \textbf{0.0[15]}\\
\hline
&&&&&&&&&&
\\
\multirow{3}{*}{40}& 0.95 & 0.6 & 23.9 & 6,522 & 64 & 0.0[5] & 72.3 & 270 & 9 & 0.0[5] & 74.8 & 192 & 8 & 0.0[5]\\
& 0.975 & 1.0 & 24.0 & 31,022 & 414 & 0.0[5] & 71.0 & 823 & 12 & 0.0[5] & 72.1 & 695 & 11 & 0.0[5]\\
& 0.99 & 1.6 & 17.6 & 122,568 & 2,907 & 0.2[3] & 73.9 & 4,416 & 37 & 0.0[5] & 75.1 & 2,543 & 26 & 0.0[5]\\
\multicolumn{3}{c|}{\textbf{Average}}&\textbf{21.8} &\textbf{ 53,371} & \textbf{1,128} & \textbf{0.1[13]} & \textbf{72.4} & \textbf{1,836} & \textbf{19} & \textbf{0.0[15]} & \textbf{74.0} & \textbf{1,143} & \textbf{15} & \textbf{0.0[15]}\\

\hline\hline
\end{tabular}
}
\end{center}
\end{table}
} 
\begin{figure}[h!]
	\centering
	\begin{subfigure}{1.0\textwidth}
		\centering
		\includegraphics[width=0.9\textwidth,trim={9cm 5.5cm 9cm 5.5cm},clip]{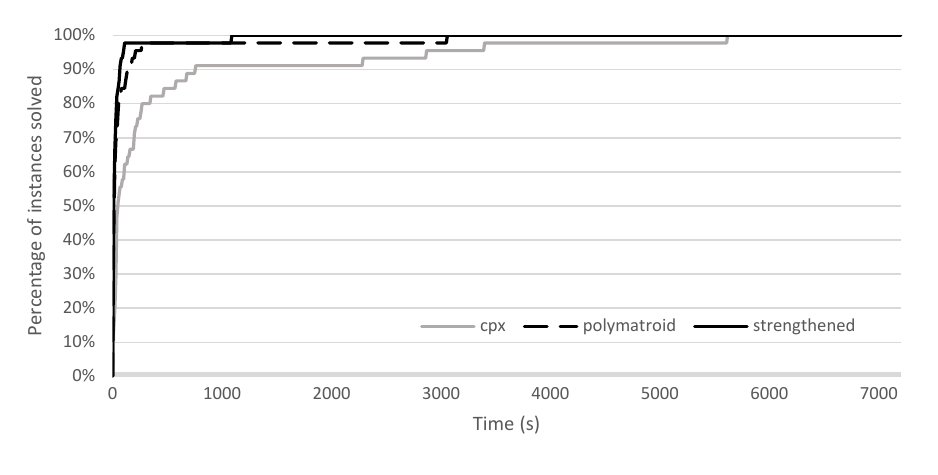}
		\caption{$\delta=0.5$.}
	\end{subfigure}
	
	\begin{subfigure}{1.0\textwidth}
		\centering
		\includegraphics[width=0.9\textwidth,trim={9cm 5.5cm 9cm 5.5cm},clip]{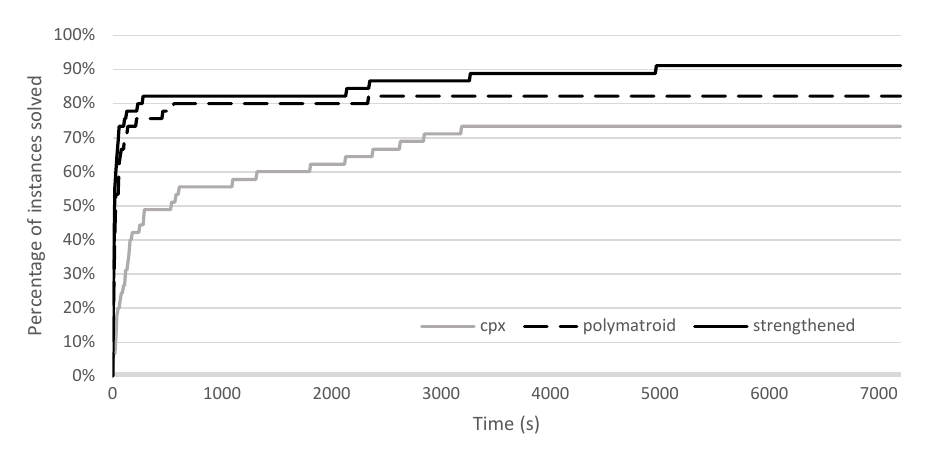}
		\caption{$\delta=1.0$.}
	\end{subfigure}
	\caption{Percentage of instances solved within a given time limit for mean-risk minimization with correlated random variables.}
	\label{fig:profileNondiag}
\end{figure}

Tables~\ref{tab:card05delta} and \ref{tab:card10delta} present the results for different \ins{values} of the diagonal dominance parameter
$\delta$. \rev{Each row represents the average over five instances generated with the same parameters and shows the initial gap (\texttt{igap}), the root gap improvement (\texttt{rimp}), the number of nodes explored (\texttt{nodes}), the time elapsed in seconds (\texttt{time}), and the end gap (\texttt{egap})[in brackets, the number of instances solved to optimality (\#)]. The initial gap is computed as $\texttt{igap}=\frac{t_{\text{opt}}-t_{\text{relax}}}{\left|t_{\text{opt}}\right|}\times 100$, where $t_{\text{opt}}$ is the objective value of the best feasible solution at termination and $t_{\text{relax}}$ is the objective value of the continuous relaxation. The end gap is computed as $\texttt{egap}=\frac{{t_\text{opt}}-t_{\text{bb}}}{\left|t_{\text{opt}}\right|}\times 100$, where $t_{\text{bb}}$ is the objective value of the best lower bound at termination.
	The root improvement is computed as $\texttt{rimp}=\frac{t_{\text{root}}-t_{\text{relax}}}{t_{\text{opt}}-t_{\text{relax}}}\times 100$, where $t_{\text{root}}$ is the value of the continuous relaxation after adding the valid inequalities to the formulation. Figure~\ref{fig:profileNondiag} shows the corresponding performance profiles.}

Observe that adding inequalities \rev{\texttt{polymatroid}} or \rev{\texttt{strengthened}} closes the initial integrality gaps by 45\% to 75\%, resulting in significant performance improvement over default CPLEX. In particular, using inequalities  \rev{\texttt{strengthened}}
for instances with $k=20$ leads to seven times speed-up with $\delta=0.5$ and two times speed-up with $\delta=1$) and lower end gaps. Moreover, for instances with $k\geq 30$ using inequalities  \rev{\texttt{strengthened}} results in at least an order-of-magnitude speed-up over default CPLEX.
The impact of both inequalities increases with higher diagonal dominance \rev{as expected}. \rev{In \del{particular, from} Figure~\ref{fig:profileNondiag} we see that for $\delta=1.0$ \texttt{cpx} requires close to 3,000 seconds to solve 70\% of the instances, while \texttt{polymatroid} requires 110 seconds and \texttt{strengthened} requires 50 seconds to solve a similar \rep{number}{quantity} of instances, i.e., \texttt{strengthened} is 50 times faster than \texttt{cpx}; in fact, \texttt{strengthened} solves in 60 seconds 73\% of the instances, the same quantity that \texttt{cpx} solves in 2 hours. Finally, we see that the \texttt{strengthened}  inequalities result in consistently better performance than the simpler \texttt{polymatroid} inequalities .}

\subsection{\rev{Conic quadratic interdiction} instances}
\label{sec:robustComputations}

In this section we test the effectiveness of the proposed inequalities \rev{for} the \rev{interdiction} problem \eqref{eq:robustConicQuadraticProblem} \rev{discussed in Section~\ref{sec:robust}}.
In our computations, we model a decision-maker that seeks a path with minimal \rev{value-at-risk}. \rev{After the decision-maker decides on a path, an adversary may attack a limited number of \del{times some} arcs on the path, increasing the expectation and/or covariance of travel times/costs.}

The feasible region $X$ is given by path constraints on a $40\times 40$ grid network. There is a potential adverse event corresponding to each arc, and each event results in an increase in the nominal duration\rev{/cost} and variance of that arc: in particular, for $i=1,\ldots,n$, $a_i\sim U[0,2]e^i$, where $e^i$ is the vector which has value $1$ in the $i$-th position and $0$ elsewhere, and the $i$-th row and column 
of $Q_i$ is drawn from $U[0,2]$ and $Q_i$ has 0 entries elsewhere. 
Each element of the nominal cost vector $a_0$ is drawn from $U[0,1]$, and the squared roots of every diagonal element of $Q_0$ are also generated from $U[0,1]$. The parameter $\Omega$ is set as in Section~\ref{sec:computationalNondiagonal}.

Table~\ref{tab:robustPath} shows the results for different values of $\alpha$ and the parameter controlling the \rev{number of attacks $\Gamma$}\rev{, and Figure~\ref{fig:profileInterdiction} shows the corresponding performance profile}. Observe that the \rev{\texttt{strengthened}} cuts  result in a better root improvement of 55\% -- compared to 30--37\% achieved by default CPLEX.
Moreover, when using the \texttt{strengthened} inequalities, 37 instances are solved to optimality, while default CPLEX is able to solve only 22 instances. We also see that in these path instances, the \rev{\texttt{polymatroid}} inequalities result in longer solution times than \texttt{cpx} (despite better root improvements). On the other hand, the \texttt{strengthened} inequalities are effective both 
in reducing the integrality gaps and solution times.

{
\renewcommand\arraystretch{0.7}
\begin{table}[h!]
\caption{Experiments with robust conic instances.}
\setlength{\tabcolsep}{2pt}
\begin{center}
\label{tab:robustPath}
\SingleSpacedXI
\scalebox{0.7}{
\begin{tabular}{ c c c |c  c c c| c c c c| c c c c}
\hline \hline
\multirow{2}{*}{$\rev{\Gamma}$} & \multirow{2}{*}{$\alpha$} & \multirow{2}{*}{\textbf{\texttt{igap}}} & \multicolumn{4}{c|}{\textbf{\texttt{cpx}}} & \multicolumn{4}{c|}{\textbf{\texttt{\rev{polymatroid}}}}& \multicolumn{4}{c}{\textbf{\texttt{\rev{strengthened}}}}\\
&&&\texttt{rimp}&\texttt{nodes}&\texttt{time}&\texttt{egap}[\#]&\texttt{rimp}&\texttt{nodes}&\texttt{time}&\texttt{egap}[\#]&\texttt{rimp}&\texttt{nodes}&\texttt{time}&\texttt{egap}[\#]\\
\hline
\multirow{3}{*}{4}& 0.95 & 22.6 & 35.1 & 65,533 & 3,124 & 0.6[4] & 44.3 & 72,322 & 5,220 & 1.2[3] & 56.8 & 17,057 & 917 & 0.0[5]\\
& 0.975 & 24.1 & 30.2 & 95,337 & 4,239 & 0.8[4] & 41.1 & 87,697 & 7,200 & 3.5[0] & 55.2 & 53,022 & 2,648 & 0.0[5]\\
& 0.99 & 25.7 & 26.6 & 153,481 & 7,200 & 2.2[0] & 37.9 & 80,160 & 7,200 & 7.6[0] & 53.5 & 102,578 & 4,452 & 0.0[5]\\
\multicolumn{3}{c|}{\textbf{Average}}&\textbf{30.6} &\textbf{ 104,117} & \textbf{4,854} & \textbf{1.2[8]} & \textbf{41.1} & \textbf{80,060} & \textbf{6,540} & \textbf{4.1[3]} & \textbf{55.2} & \textbf{57,552} & \textbf{2,672} & \textbf{0.0[15]}\\
\hline
&&&&&&&&&&
\\
\multirow{3}{*}{6}& 0.95 & 26.9 & 38.8 & 73,898 & 3,422 & 0.4[4] & 45.0 & 89,319 & 5,771 & 2.0[3] & 56.2 & 33,364 & 1,644 & 0.0[5]\\
& 0.975 & 28.1 & 34.6 & 138,231 & 5,676 & 1.8[2] & 41.3 & 96,917 & 7,200 & 5.5[0] & 54.1 & 113,745 & 4,895 & 0.0[5]\\
& 0.99 & 29.7 & 32.0 & 160,074 & 6,823 & 4.2[1] & 38.9 & 94,762 & 7,200 & 7.2[0] & 52.2 & 113,954 & 6,091 & 2.0[1]\\
\multicolumn{3}{c|}{\textbf{Average}}&\textbf{35.1} &\textbf{ 124,068} & \textbf{5,307} & \textbf{2.1[7]} & \textbf{41.7} & \textbf{93,666} & \textbf{6,704} & \textbf{4.9[2]} & \textbf{54.2} & \textbf{87,021} & \textbf{4,210} & \textbf{0.7[11]}\\
\hline
&&&&&&&&&&
\\
\multirow{3}{*}{8}& 0.95 & 30.2 & 40.9 & 143,946 & 5,474 & 0.8[4] & 46.4 & 112,279 & 6,822 & 1.6[1] & 55.2 & 53,942 & 2,234 & 0.0[5]\\
& 0.975 & 31.3 & 36.3 & 145,582 & 5,967 & 1.9[2] & 42.7 & 107,432 & 7,200 & 4.8[0] & 53.4 & 99,904 & 4,679 & 0.4[4]\\
& 0.99 & 32.7 & 34.2 & 123,325 & 6,512 & 3.5[1] & 39.5 & 94,691 & 7,200 & 8.2[0] & 51.1 & 136,632 & 6,162 & 2.4[2]\\
\multicolumn{3}{c|}{\textbf{Average}}&\textbf{37.1} &\textbf{ 137,618} & \textbf{5,984} & \textbf{2.1[7]} & \textbf{42.8} & \textbf{104,801} & \textbf{7,055} & \textbf{4.9[1]} & \textbf{53.2} & \textbf{96,826} & \textbf{4,358} & \textbf{0.9[11]}\\
\hline\hline
\end{tabular}
}
\end{center}
\end{table}
} 

	\begin{figure}[!h]
	\centering
	\includegraphics[width=0.9\textwidth,trim={9cm 5.5cm 9cm 5.5cm},clip]{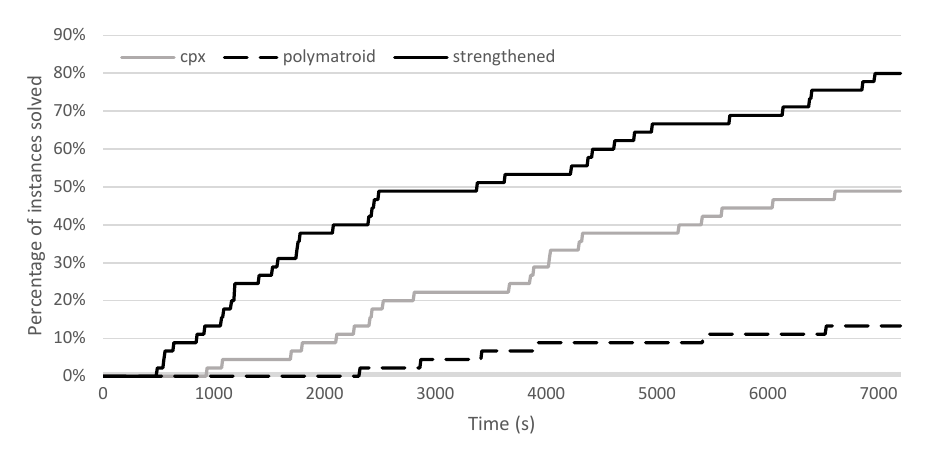}
	\caption{Percentage of instances solved within a given time limit for interdiction problems.}
	\label{fig:profileInterdiction}
\end{figure}
\ignore{
Additionally, Figure~\ref{fig:robustFigure} shows, for different values of the robustness parameter $b$, the objective value of the approximate robust counterpart $\omega_a$ and the lower bound on the exact robust problem $\omega_\ell$, defined in \eqref{eq:defOmega} and \eqref{eq:defLB}, respectively. The figure also shows the nominal objective value corresponding to the robust solution $\bar \omega_n$, i.e., $\bar \omega_n=a_0'\bar{x} +\sqrt{\bar x Q_0 \bar x}$ where $\bar x$ is an optimal solution of (RC).  We see that the nominal values $\bar \omega_n$, corresponding to the objective value if none of the events are realized, are not substantially affected by changes in $b$.
The computations suggest that the approximate robust formulation (RC) protects against adverse scenarios without leading to overly conservative solutions.

\begin{figure}[h!]
	\label{fig:robustFigure}
	\includegraphics[scale = 0.9, trim={8cm 5cm 8cm 4cm},clip,width=0.75\textwidth]{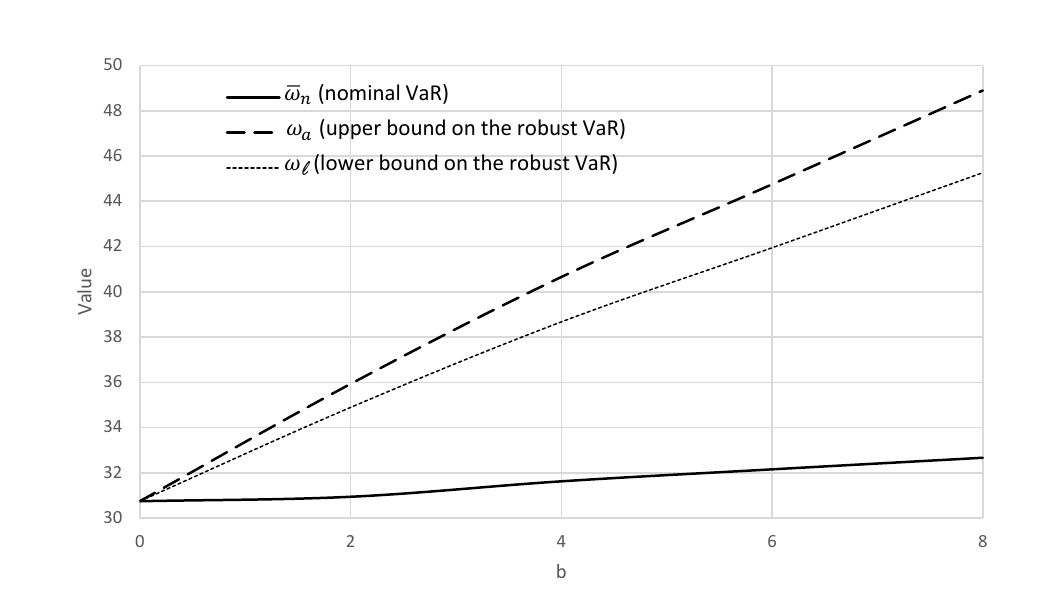}
		\caption{Nominal and robust objective values as a function of $k$.}
\end{figure}
}

\subsection{Binary fractional optimization instances}
\label{sec:fractionaComputations}

We \rev{now} test the inequalities in a binary fractional problem arising in assortment optimization with cardinality constraint:
\begin{align*}
\text{(FP)} \ \ \max\left \{ \sum_{j=1}^m\frac{\sum_{i=1}^nc_{ij}x_i}{a_{0j}+\sum_{i=1}^na_{ij}x_i} \ : \
 \sum_{i=1}^nx_i\leq k, \
x\in \{0,1\}^n \right \} \cdot
\end{align*}
The data is generate\rev{d} as in the assortment optimization problems considered in \cite{Sen2015}: $a_{ij}\sim U[0,1]$ for all $i,j$, $c_{ij}=a_{ij}r_{ij}$ with $r_{ij}\sim U[1,3]$, $n=200$, $m=20$ and $a_{0j}=a_0$ for all $j=1,\ldots,m$ with $a_0\in \{5,10\}$, and $k\in \{10,20,50\}$. 

Binary fractional problems (FP) are usually solved by linearizing the fractional terms \citep[see][]{Tawarmalani2002,Prokopyev2005,Bront2009,Mendez2014,Sen2015,Borrero2016b}, which requires the addition of $O(nm)$ additional variables and big-M constraints. 
On the other hand, the rotated cone reformulation outlined in Section~\ref{sec:fractional}, requires adding only $m$ additional variables and avoids big-M constraints altogether. 

We test the classical big-M linear formulation used in \cite{Bront2009,Mendez2014} (\texttt{cpx-milo}), the conic formulation without adding inequalities (\texttt{cpx-conic}) and the conic formulation strengthened with \rev{\texttt{polymatroid}} inequalities \footnote{For these instances the \rev{\texttt{strengthened}} inequalities perform very similarly to \rev{\texttt{polymatroid}, since the simpler inequalities already achieve close to 100\% root gap improvements}. Therefore, we only present the results with inequalities \rev{\texttt{polymatroid}}.}. Table~\ref{tab:assortment200} shows the results. Each row represents the average over five instances generated with the same parameters and for each combination of the parameters $a_0$ and $k$ and for each formulation, the root gap (\texttt{rgap}),  the number of nodes explored (\texttt{nodes}), the time elapsed in seconds (\texttt{time}), and the end gap (\texttt{egap})[in brackets, the number of instances solved to optimality (\#)]. The root gap is computed as $\texttt{rgap}=\frac{t_{\text{opt}}-t_{\text{root}}}{\left|t_{\text{opt}}\right|}\times 100$, where $t_{\text{opt}}$ is the objective value of the best feasible solution at termination, and $t_{\text{root}}$ is the objective value of the relaxation obtained after processing the root node (i.e., after user cuts and cuts added by CPLEX). 

{
		\renewcommand\arraystretch{0.7}
		\begin{table}[t!]
			\caption{Experiments with binary fractional optimization.}
			\setlength{\tabcolsep}{2pt}
			\begin{center}
				\label{tab:assortment200}
				\SingleSpacedXI
				\scalebox{0.75}{
					\begin{tabular}{  c c |l  l l l| l l l l| l l l l}
						\hline \hline
						\multirow{2}{*}{\texttt{$a_0$}} &\multirow{2}{*}{\texttt{$k$}} & \multicolumn{4}{c|}{\textbf{\texttt{ \rev{cpx-}milo}}}& \multicolumn{4}{c|}{\textbf{\texttt{ \rev{cpx-conic}}}}&  \multicolumn{4}{c}{\textbf{\texttt{ \rev{polymatroid}}}}\\
						&&\texttt{rgap}&\texttt{nodes}&\texttt{time}&\texttt{egap[\#]}&\texttt{rgap}&\texttt{nodes}&\texttt{time}&\texttt{egap[\#]}&\texttt{rgap}&\texttt{nodes}&\texttt{time}&\texttt{egap[\#]}\\
						\hline
						\multirow{3}{*}{5} & 10  & 50.9 & 20,737 & 7,200 & 43.6[0] & 3.1 & 24,073 & 572 & 0.0[5] & 0.1 & 46 & 19 & 0.0[5]\\
						& 20 & 18.0 & 51,180 & 7,200 & 17.0[0] & 2.7 & 123,655 & 7,200 & 1.9[0] & 0.0 & 118 & 54 & 0.0[5] \\
						& 50 & 0.9 & 621,742 & 6,010 & 0.5[1] & 4.9 & 55,155 & 7,200 & 4.5[0]& 0.1 & 15,465 & 263 & 0.0[5] \\
						\multicolumn{2}{c|}{\textbf{Average}}&\textbf{23.3} &\textbf{ 231,220} & \textbf{6,803} & \textbf{20.4[1]} & \textbf{3.2} & \textbf{67,628} & \textbf{4,991} & \textbf{2.1[5]}& \textbf{0.1} & \textbf{5,210} & \textbf{112} & \textbf{0.0[15]} \\
						\hline
						&&&&&&&&&&&&&\\
						\multirow{3}{*}{10} & 10  & 46.8 & 380,700 & 7,200 & 15.9[0]  & 2.2 & 48,541 & 972 & 0.0[5] & 0.0 & 6 & 14 & 0.0[5]  \\
						& 20 & 39.8 & 23,770 & 7,200 & 37.4[0] & 3.7 & 206,603 & 7,200 & 1.4[0]& 0.0 & 61 & 37 & 0.0[5] \\
						& 50 & 5.6 & 136,382 & 7,200 & 5.2[0]& 5.1 & 52,700 & 7,200 & 4.6[0]  & 0.1 & 36,959 & 396 & 0.0[5]\\
						\multicolumn{2}{c|}{\textbf{Average}}&\textbf{30.7} &\textbf{ 180,284} & \textbf{7,200} & \textbf{19.5[0]} & \textbf{4.3} & \textbf{102,615} & \textbf{5,124} & \textbf{2.0[5]}& \textbf{0.0} & \textbf{12,342} & \textbf{149} & \textbf{0.0[15]}\\
						
						\hline\hline
					\end{tabular}
 				}
			\end{center}
		\end{table}
}

We see that the conic formulation with polymatroid inequalities results in substantially faster solution times than the other formulations. In particular, CPLEX with the classical big-M linear optimization formulation \texttt{\rev{cpx-}milo} can only solve 1/30 instances after two hours of branch and bound, and the average end gaps are 20\%; the conic formulation with extended polymatroid cuts is able to solve all instances to optimality in less than 3 minutes (on average). We see that root gaps for \texttt{\rev{polymatroid}} are very small in all instances (less than 0.1\%), and optimality can be proven in instances with small cardinality parameter $k$ after few branch-and-bound nodes (e.g., in instances with $k=10$ and $a_0=5$ optimality is proven after 46 nodes, while \texttt{\rev{cpx-conic}} requires 24,000 nodes to prove optimality). 

\section{Conclusions}
\label{sec:conclusions}


We propose new convex valid inequalities \rev{that exploit submodularity} for conic quadratic mixed 0-1 sets. 
The studied sets arise in a variety of risk-adverse decision-making problems (e.g, chance constrained optimization with correlated variables, robust optimization with ellipsoidal or discrete uncertainty sets) as well as in models of other problems commonly arising in operations research (e.g., lot sizing, scheduling, assortment, fractional linear optimization). \rev{The unbounded version of the convex inequalities, which arise naturally in most applications, can be efficiently implemented as linear cuts in
an extended space, which make them particularly effective. Moreover, the inequalities can be strengthened to take advantage of other constraints in a problem through approximate lifting without 
affecting this convenient property.
Computational experiments performed on correlated mean-risk minimization, robust interdiction
and assortment optimization problems indicate that the proposed inequalities improve the performance of branch-and-bound solvers substantially; in some cases, problems for which no efficient algorithms were known are now solved in seconds.}

\ignore{
The inequalities are derived by exploiting partial submodularity arising from the binary variables and generalize the polymatroid inequalities known for the binary case. They completely describe the convex hull of a single conic quadratic constraint as well as rotated cone constraint with unbounded continuous variables. We also show how to strengthen the inequalities with additional problem constraints. The experiments on testing the inequalities for different classes of problems indicate that in all cases the inequalities 
}

\section*{Acknowledgment}
A. Atamt{\"u}rk is supported, in part,
by grant FA9550-10-1-0168 from the Office of the Assistant Secretary of Defense for Research and Engineering.
\rev{A. G\'omez is supported, in part, by the National Science Foundation under Grant No. 1818700.} 

\bibliographystyle{apalike}
\bibliography{Bibliography}

\begin{thebibliography}{}

\bibitem[Ahmed and Atamt{\"u}rk, 2011]{aa:max-submodular}
Ahmed, S. and Atamt{\"u}rk, A. (2011).
\newblock Maximizing a class of submodular utility functions.
\newblock {\em Mathematical Programming}, 128:149--169.

\bibitem[Akt{\"u}rk et~al., 2009]{akturk2009strong}
Akt{\"u}rk, M.~S., Atamt{\"u}rk, A., and G{\"u}rel, S. (2009).
\newblock A strong conic quadratic reformulation for machine-job assignment
  with controllable processing times.
\newblock {\em Operations Research Letters}, 37:187--191.

\bibitem[Akt{\"u}rk et~al., 2010]{aag:scheduling}
Akt{\"u}rk, M.~S., Atamt{\"u}rk, A., and G{\"u}rel, S. (2010).
\newblock Parallel machine match-up scheduling with manufacturing cost
  considerations.
\newblock {\em Journal of Scheduling}, 13:95--110.

\bibitem[Amaldi et~al., 2011]{Amaldi2011}
Amaldi, E., Bosio, S., Malucelli, F., and Yuan, D. (2011).
\newblock Solving nonlinear covering problems arising in wlan design.
\newblock {\em Operations Research}, 59(1):173--187.

\bibitem[Amiri, 1997]{amiri1997solution}
Amiri, A. (1997).
\newblock Solution procedures for the service system design problem.
\newblock {\em Computers \& Operations Research}, 24:49--60.

\bibitem[Anstreicher, 2012]{anstreicher2012convex}
Anstreicher, K.~M. (2012).
\newblock On convex relaxations for quadratically constrained quadratic
  programming.
\newblock {\em Mathematical Programming}, 136:233--251.

\bibitem[Atamt{\"u}rk et~al., 2012]{Atamturk2012}
Atamt{\"u}rk, A., Berenguer, G., and Shen, Z.-J. (2012).
\newblock A conic integer programming approach to stochastic joint
  location-inventory problems.
\newblock {\em Operations Research}, 60:366--381.

\bibitem[Atamt{\"u}rk and Bhardwaj, 2015]{ab:submodular-covering}
Atamt{\"u}rk, A. and Bhardwaj, A. (2015).
\newblock Supermodular covering knapsack polytope.
\newblock {\em Discrete Optimization}, 18:74--86.

\bibitem[Atamt{\"u}rk et~al., 2017]{ADJ:mr-interdiction}
Atamt{\"u}rk, A., Deck, C., and Jeon, H. (2017).
\newblock Successive quadratic upper-bounding for discrete mean-risk
  minimization and network interdiction.
\newblock {\em arXiv preprint arXiv:1708.02371}.
\newblock BCOL Reseach Report 17.05, UC Berkeley.

\bibitem[Atamt{\"u}rk and G{\'o}mez, 2017]{Atamturk2017}
Atamt{\"u}rk, A. and G{\'o}mez, A. (2017).
\newblock Maximizing a class of utility functions over the vertices of a
  polytope.
\newblock {\em Operations Research}, 65:433--445.

\bibitem[Atamt{\"u}rk and G{\'o}mez, 2018]{atamturk2018strong}
Atamt{\"u}rk, A. and G{\'o}mez, A. (2018).
\newblock Strong formulations for quadratic optimization with {M}-matrices and
  indicator variables.
\newblock {\em Mathematical Programming}, 170:141--176.

\bibitem[Atamt{\"u}rk and Jeon, 2017]{AJ:conicvub}
Atamt{\"u}rk, A. and Jeon, H. (2017).
\newblock Lifted polymatroid inequalities for mean-risk optimization with
  indicator variables.
\newblock {\em arXiv preprint arXiv:1705.05915}.
\newblock BCOL Research Report 17.01, UC Berkeley.

\bibitem[Atamt{\"u}rk and Narayanan, 2007]{AN:conic-mir:ipco}
Atamt{\"u}rk, A. and Narayanan, V. (2007).
\newblock Cuts for conic mixed-integer programming.
\newblock In {\em International Conference on Integer Programming and
  Combinatorial Optimization}, pages 16--29. Springer.

\bibitem[Atamt{\"u}rk and Narayanan, 2008]{Atamturk2008b}
Atamt{\"u}rk, A. and Narayanan, V. (2008).
\newblock Polymatroids and mean-risk minimization in discrete optimization.
\newblock {\em Operations Research Letters}, 36:618--622.

\bibitem[Atamt{\"u}rk and Narayanan, 2009]{an:submodular-knapsack}
Atamt{\"u}rk, A. and Narayanan, V. (2009).
\newblock The submodular knapsack polytope.
\newblock {\em Discrete Optimization}, 6:333--344.

\bibitem[Atamt{\"u}rk and Narayanan, 2010]{Atamturk2010}
Atamt{\"u}rk, A. and Narayanan, V. (2010).
\newblock Conic mixed-integer rounding cuts.
\newblock {\em Mathematical Programming}, 122:1--20.

\bibitem[Atamt{\"u}rk and Narayanan, 2011]{Atamturk2011}
Atamt{\"u}rk, A. and Narayanan, V. (2011).
\newblock Lifting for conic mixed-integer programming.
\newblock {\em Mathematical Programming}, 126:351--363.

\bibitem[Belotti et~al., 2015]{Belotti2015}
Belotti, P., G{\'o}ez, J.~C., P{\'o}lik, I., Ralphs, T.~K., and Terlaky, T.
  (2015).
\newblock A conic representation of the convex hull of disjunctive sets and
  conic cuts for integer second order cone optimization.
\newblock In {\em Numerical Analysis and Optimization}, pages 1--35. Springer.

\bibitem[Ben-Tal et~al., 2009a]{RO-book}
Ben-Tal, A., El~Ghaoui, L., and Nemirovski, A. (2009a).
\newblock {\em Robust Optimization}.
\newblock Princeton Series in Applied Mathematics. Princeton University Press.

\bibitem[Ben-Tal et~al., 2009b]{ben:robust-book}
Ben-Tal, A., El~Ghaoui, L., and Nemirovski, A. (2009b).
\newblock {\em Robust optimization}.
\newblock Princeton University Press.

\bibitem[Ben-Tal and Nemirovski, 1998]{BenTal1998}
Ben-Tal, A. and Nemirovski, A. (1998).
\newblock Robust convex optimization.
\newblock {\em Mathematics of Operations Research}, 23:769--805.

\bibitem[Ben-Tal and Nemirovski, 1999]{BenTal1999}
Ben-Tal, A. and Nemirovski, A. (1999).
\newblock Robust solutions of uncertain linear programs.
\newblock {\em Operations Research Letters}, 25:1--13.

\bibitem[Ben-Tal and Nemirovski, 2001]{BenTal2001}
Ben-Tal, A. and Nemirovski, A. (2001).
\newblock On polyhedral approximations of the second-order cone.
\newblock {\em Mathematics of Operations Research}, 26:193--205.

\bibitem[Berman and Krass, 2001]{berman200111}
Berman, O. and Krass, D. (2001).
\newblock 11 facility location problems with stochastic demands and congestion.
\newblock {\em Facility Location: Applications and Theory}, page 329.

\bibitem[Bertsimas and Sim, 2003]{Bertsimas2003}
Bertsimas, D. and Sim, M. (2003).
\newblock Robust discrete optimization and network flows.
\newblock {\em Mathematical Programming}, 98:49--71.

\bibitem[Bertsimas and Sim, 2004]{Bertsimas2004}
Bertsimas, D. and Sim, M. (2004).
\newblock The price of robustness.
\newblock {\em Operations Research}, 52:35--53.

\bibitem[Birge and Louveaux, 2011]{BL:sp-book}
Birge, J.~R. and Louveaux, F. (2011).
\newblock {\em {Introduction to Stochastic Programming}}.
\newblock Springer Science \& Business Media.

\bibitem[Bollapragada and Rao, 1999]{Bollapragada1999}
Bollapragada, R. and Rao, U. (1999).
\newblock Single-stage resource allocation and economic lot scheduling on
  multiple, nonidentical production lines.
\newblock {\em Management Science}, 45:889--904.

\bibitem[Bonami, 2011]{Bonami2011}
Bonami, P. (2011).
\newblock Lift-and-project cuts for mixed integer convex programs.
\newblock In {\em International Conference on Integer Programming and
  Combinatorial Optimization}, pages 52--64. Springer.

\bibitem[Borrero et~al., 2016a]{Borrero2016}
Borrero, J.~S., Gillen, C., and Prokopyev, O.~A. (2016a).
\newblock Fractional 0--1 programming: applications and algorithms.
\newblock {\em Journal of Global Optimization}, pages 1--28.

\bibitem[Borrero et~al., 2016b]{Borrero2016b}
Borrero, J.~S., Gillen, C., and Prokopyev, O.~A. (2016b).
\newblock A simple technique to improve linearized reformulations of fractional
  (hyperbolic) 0--1 programming problems.
\newblock {\em Operations Research Letters}, 44:479--486.

\bibitem[Bront et~al., 2009]{Bront2009}
Bront, J. J.~M., M{\'e}ndez-D{\'\i}az, I., and Vulcano, G. (2009).
\newblock A column generation algorithm for choice-based network revenue
  management.
\newblock {\em Operations Research}, 57:769--784.

\bibitem[Bulut and Tasgetiren, 2014]{Bulut2014}
Bulut, O. and Tasgetiren, M.~F. (2014).
\newblock An artificial bee colony algorithm for the economic lot scheduling
  problem.
\newblock {\em International Journal of Production Research}, 52:1150--1170.

\bibitem[Burer and K{\i}l{\i}n{\c{c}}-Karzan, 2017]{Burer2017}
Burer, S. and K{\i}l{\i}n{\c{c}}-Karzan, F. (2017).
\newblock How to convexify the intersection of a second order cone and a
  nonconvex quadratic.
\newblock {\em Mathematical Programming}, 162:393--429.

\bibitem[Castro et~al., 2005]{Castro2005}
Castro, P.~M., Barbosa-P{\'o}voa, A.~P., and Novais, A.~Q. (2005).
\newblock Simultaneous design and scheduling of multipurpose plants using
  resource task network based continuous-time formulations.
\newblock {\em Industrial \& Engineering Chemistry Research}, 44:343--357.

\bibitem[Castro et~al., 2009]{Castro2009}
Castro, P.~M., Westerlund, J., and Forssell, S. (2009).
\newblock Scheduling of a continuous plant with recycling of byproducts: A case
  study from a tissue paper mill.
\newblock {\em Computers \& Chemical Engineering}, 33:347--358.

\bibitem[Ceria and Soares, 1999]{CS:dis-conv}
Ceria, S. and Soares, J. (1999).
\newblock Convex programming for disjunctive convex optimization.
\newblock {\em Mathematical Programming}, 86:595--614.

\bibitem[{\c{C}}ezik and Iyengar, 2005]{Cezik2005}
{\c{C}}ezik, M.~T. and Iyengar, G. (2005).
\newblock Cuts for mixed 0-1 conic programming.
\newblock {\em Mathematical Programming}, 104:179--202.

\bibitem[Cormican et~al., 1998]{cormican1998stochastic}
Cormican, K.~J., Morton, D.~P., and Wood, R.~K. (1998).
\newblock Stochastic network interdiction.
\newblock {\em Operations Research}, 46:184--197.

\bibitem[Dadush et~al., 2011a]{Dadush2011b}
Dadush, D., Dey, S., and Vielma, J. (2011a).
\newblock On the {Chv{\'a}tal-Gomory} closure of a compact convex set.
\newblock {\em Integer Programming and Combinatoral Optimization}, pages
  130--142.

\bibitem[Dadush et~al., 2011b]{Dadush2011c}
Dadush, D., Dey, S.~S., and Vielma, J.~P. (2011b).
\newblock The chv{\'a}tal-gomory closure of a strictly convex body.
\newblock {\em Mathematics of Operations Research}, 36:227--239.

\bibitem[Dadush et~al., 2011c]{Dadush2011}
Dadush, D., Dey, S.~S., and Vielma, J.~P. (2011c).
\newblock The split closure of a strictly convex body.
\newblock {\em Operations Research Letters}, 39:121--126.

\bibitem[Davidoff et~al., 2003]{Davidoff2003}
Davidoff, G., Sarnak, P., and Valette, A. (2003).
\newblock {\em Elementary Number Theory, Group Theory and Ramanujan Graphs},
  volume~55.
\newblock Cambridge University Press.

\bibitem[D{\'e}sir et~al., 2014]{desir2014near}
D{\'e}sir, A., Goyal, V., and Zhang, J. (2014).
\newblock Near-optimal algorithms for capacity constrained assortment
  optimization.

\bibitem[Edmonds, 1970]{Edmonds1970}
Edmonds, J. (1970).
\newblock Submodular functions, matroids, and certain polyhedra.
\newblock In Guy, R., Hanani, H., Sauer, N., and Sch\"onenheim, J., editors,
  {\em Combinatorial Structures and Their Applications}, pages 69--87. Gordon
  and Breach.

\bibitem[El~Ghaoui et~al., 2003]{ghaoui2003worst}
El~Ghaoui, L., Oks, M., and Oustry, F. (2003).
\newblock Worst-case value-at-risk and robust portfolio optimization: A conic
  programming approach.
\newblock {\em Operations Research}, 51:543--556.

\bibitem[Elhedhli, 2005]{elhedhli2005exact}
Elhedhli, S. (2005).
\newblock Exact solution of a class of nonlinear knapsack problems.
\newblock {\em Operations Research Letters}, 33:615--624.

\bibitem[Elhedhli, 2006]{elhedhli2006service}
Elhedhli, S. (2006).
\newblock Service system design with immobile servers, stochastic demand, and
  congestion.
\newblock {\em Manufacturing \& Service Operations Management}, 8:92--97.

\bibitem[Fujishige, 2005]{F:submodularBook}
Fujishige, S. (2005).
\newblock {\em {Submodular Functions and Optimization}}, volume~58.
\newblock Elsevier.

\bibitem[Gilmore and Gomory, 1963]{Gilmore1963}
Gilmore, P.~C. and Gomory, R.~E. (1963).
\newblock A linear programming approach to the cutting stock problem—part ii.
\newblock {\em Operations Research}, 11:863--888.

\bibitem[G\'omez, 2018]{gomez2018strong}
G\'omez, A. (2018).
\newblock Strong formulations for conic quadratic optimization with indicator
  variables.
\newblock {\em http://www.optimization-online.org/DB\_HTML/2018/05/6616.html}.

\bibitem[Gr{\"o}tschel et~al., 1981]{GLS:ellipsoid}
Gr{\"o}tschel, M., Lov{\'a}sz, L., and Schrijver, A. (1981).
\newblock The ellipsoid method and its consequences in combinatorial
  optimization.
\newblock {\em Combinatorica}, 1:169--197.

\bibitem[G{\"u}nl{\"u}k and Linderoth, 2010]{Gunluk2010}
G{\"u}nl{\"u}k, O. and Linderoth, J. (2010).
\newblock Perspective reformulations of mixed integer nonlinear programs with
  indicator variables.
\newblock {\em Mathematical Programming}, 124:183--205.

\bibitem[Hijazi et~al., 2013]{Hijazi2013}
Hijazi, H., Bonami, P., and Ouorou, A. (2013).
\newblock An outer-inner approximation for separable mixed-integer nonlinear
  programs.
\newblock {\em INFORMS Journal on Computing}, 26:31--44.

\bibitem[Hochbaum, 2010]{Hochbaum2010}
Hochbaum, D.~S. (2010).
\newblock Polynomial time algorithms for ratio regions and a variant of
  normalized cut.
\newblock {\em IEEE Transactions on Pattern Analysis and Machine Intelligence},
  32:889--898.

\bibitem[Hochbaum et~al., 2013]{Hochbaum2013}
Hochbaum, D.~S., Lyu, C., and Bertelli, E. (2013).
\newblock Evaluating performance of image segmentation criteria and techniques.
\newblock {\em EURO Journal on Computational Optimization}, 1:155--180.

\bibitem[Israeli and Wood, 2002]{israeli2002shortest}
Israeli, E. and Wood, R.~K. (2002).
\newblock Shortest-path network interdiction.
\newblock {\em Networks}, 40:97--111.

\bibitem[K{\i}l{\i}n{\c{c}} et~al., 2010]{Kilinc2010}
K{\i}l{\i}n{\c{c}}, M., Linderoth, J., and Luedtke, J. (2010).
\newblock Effective separation of disjunctive cuts for convex mixed integer
  nonlinear programs.
\newblock {\em Optimization Online}.

\bibitem[K{\i}l{\i}n{\c{c}}-Karzan, 2015]{Kilincc2015}
K{\i}l{\i}n{\c{c}}-Karzan, F. (2015).
\newblock On minimal valid inequalities for mixed integer conic programs.
\newblock {\em Mathematics of Operations Research}, 41:477--510.

\bibitem[K{\i}l{\i}n{\c{c}}-Karzan and Y{\i}ld{\i}z, 2015]{KilincKarzan2015}
K{\i}l{\i}n{\c{c}}-Karzan, F. and Y{\i}ld{\i}z, S. (2015).
\newblock Two-term disjunctions on the second-order cone.
\newblock {\em Mathematical Programming}, 154:463--491.

\bibitem[Lim and Smith, 2007]{lim2007algorithms}
Lim, C. and Smith, J.~C. (2007).
\newblock Algorithms for discrete and continuous multicommodity flow network
  interdiction problems.
\newblock {\em {IIE} Transactions}, 39:15--26.

\bibitem[Lov{\'a}sz, 1983]{L:submodular-convex}
Lov{\'a}sz, L. (1983).
\newblock Submodular functions and convexity.
\newblock In Bachem, A., Korte, B., and Gr{\"o}tschel, M., editors, {\em
  Mathematical Programming The State of the Art: Bonn 1982}, pages 235--257,
  Berlin, Heidelberg. Springer.

\bibitem[Lubin et~al., 2016]{Lubin2016}
Lubin, M., Yamangil, E., Bent, R., and Vielma, J.~P. (2016).
\newblock Polyhedral approximation in mixed-integer convex optimization.
\newblock {\em arXiv preprint arXiv:1607.03566}.

\bibitem[M{\'e}ndez-D{\'\i}az et~al., 2014]{Mendez2014}
M{\'e}ndez-D{\'\i}az, I., Miranda-Bront, J.~J., Vulcano, G., and Zabala, P.
  (2014).
\newblock A branch-and-cut algorithm for the latent-class logit assortment
  problem.
\newblock {\em Discrete Applied Mathematics}, 164:246--263.

\bibitem[Modaresi et~al., 2016]{Modaresi2016}
Modaresi, S., K{\i}l{\i}n{\c{c}}, M.~R., and Vielma, J.~P. (2016).
\newblock Intersection cuts for nonlinear integer programming: Convexification
  techniques for structured sets.
\newblock {\em Mathematical Programming}, 155:575--611.

\bibitem[Modaresi and Vielma, 2014]{Modaresi2014}
Modaresi, S. and Vielma, J.~P. (2014).
\newblock Convex hull of two quadratic or a conic quadratic and a quadratic
  inequality.
\newblock {\em Mathematical Programming}, pages 1--27.

\bibitem[Nahmias, 2001]{Nahmias2001}
Nahmias, S. (2001).
\newblock {\em Production and Operations Analysis}.
\newblock McGraw Hill.

\bibitem[Nikolova et~al., 2006]{Nikolova2006}
Nikolova, E., Kelner, J., Brand, M., and Mitzenmacher, M. (2006).
\newblock Stochastic shortest paths via quasi-convex maximization.
\newblock {\em Algorithms--ESA 2006}, pages 552--563.

\bibitem[Orlin, 2009]{O:faster}
Orlin, J.~B. (2009).
\newblock A faster strongly polynomial time algorithm for submodular function
  minimization.
\newblock {\em Mathematical Programming}, 118:237--251.

\bibitem[{\"O}zsen et~al., 2008]{Ozsen2008}
{\"O}zsen, L., Coullard, C.~R., and Daskin, M.~S. (2008).
\newblock Capacitated warehouse location model with risk pooling.
\newblock {\em Naval Research Logistics}, 55:295--312.

\bibitem[Pesenti and Ukovich, 2003]{Pesenti2003}
Pesenti, R. and Ukovich, W. (2003).
\newblock Economic lot scheduling on multiple production lines with resource
  constraints.
\newblock {\em International Journal of Production Economics}, 81:469--481.

\bibitem[Poljak and Wolkowicz, 1995]{poljak1995convex}
Poljak, S. and Wolkowicz, H. (1995).
\newblock Convex relaxations of (0, 1)-quadratic programming.
\newblock {\em Mathematics of Operations Research}, 20:550--561.

\bibitem[Prokopyev et~al., 2009]{Prokopyev2009}
Prokopyev, O.~A., Kong, N., and Martinez-Torres, D.~L. (2009).
\newblock The equitable dispersion problem.
\newblock {\em European Journal of Operational Research}, 197:59--67.

\bibitem[Prokopyev et~al., 2005]{Prokopyev2005}
Prokopyev, O.~A., Meneses, C., Oliveira, C.~A., and Pardalos, P.~M. (2005).
\newblock On multiple-ratio hyperbolic 0--1 programming problems.
\newblock {\em Pacific Journal of Optimization}, 1:327--345.

\bibitem[Sahinidis and Grossmann, 1991]{Sahinidis1991}
Sahinidis, N. and Grossmann, I.~E. (1991).
\newblock Minlp model for cyclic multiproduct scheduling on continuous parallel
  lines.
\newblock {\em Computers \& Chemical Engineering}, 15:85--103.

\bibitem[Santana and Dey, 2017]{Santana2017}
Santana, A. and Dey, S.~S. (2017).
\newblock Some cut-generating functions for second-order conic sets.
\newblock {\em Discrete Optimization}, 24:51--65.

\bibitem[Schrijver, 2000]{S:combinatorial}
Schrijver, A. (2000).
\newblock A combinatorial algorithm minimizing submodular functions in strongly
  polynomial time.
\newblock {\em Journal of Combinatorial Theory, Series B}, 80:346--355.

\bibitem[{\c S}en et~al., 2015]{Sen2015}
{\c S}en, A., Atamt{\"u}rk, A., and Kaminsky, P. (2015).
\newblock A conic integer programming approach to constrained assortment
  optimization under the mixed multinomial logit model.
\newblock {\em arXiv preprint arXiv:1705.09040}.
\newblock BCOL Research Report 15.06, UC Berkeley, Forthcoming in Operations
  Research.

\bibitem[Sharpe, 1994]{Sharpe1994}
Sharpe, W.~F. (1994).
\newblock The {S}harpe ratio.
\newblock {\em The Journal of Portfolio Management}, 21:49--58.

\bibitem[Shen et~al., 2003]{shen2003}
Shen, Z.-J.~M., Coullard, C., and Daskin, M.~S. (2003).
\newblock A joint location-inventory model.
\newblock {\em Transportation Science}, 37:40--55.

\bibitem[Stubbs and Mehrotra, 1999]{SM:lift-project}
Stubbs, A.~R. and Mehrotra, S. (1999).
\newblock A branch-and-cut method for 0-1 mixed convex programming.
\newblock {\em Mathematical Programming}, 86:515--532.

\bibitem[Tawarmalani et~al., 2002]{Tawarmalani2002}
Tawarmalani, M., Ahmed, S., and Sahinidis, N.~V. (2002).
\newblock Global optimization of 0-1 hyperbolic programs.
\newblock {\em Journal of Global Optimization}, 24:385--416.

\bibitem[Tawarmalani and Sahinidis, 2005]{Tawarmalani2005}
Tawarmalani, M. and Sahinidis, N.~V. (2005).
\newblock A polyhedral branch-and-cut approach to global optimization.
\newblock {\em Mathematical Programming}, 103:225--249.

\bibitem[Wood, 1993]{wood1993deterministic}
Wood, R.~K. (1993).
\newblock Deterministic network interdiction.
\newblock {\em Mathematical and Computer Modelling}, 17:1--18.

\bibitem[Yu and Ahmed, 2017]{Yu2015}
Yu, J. and Ahmed, S. (2017).
\newblock Polyhedral results for a class of cardinality constrained submodular
  minimization problems.
\newblock {\em Discrete Optimization}, 24:87--102.

\bibitem[Zhang et~al., 2016]{zhang2016ambiguous}
Zhang, Y., Jiang, R., and Shen, S. (2016).
\newblock Ambiguous chance-constrained bin packing under mean-covariance
  information.
\newblock {\em arXiv preprint arXiv:1610.00035}.

\end{thebibliography}

\appendix

\section{}
\label{sec:appendix}

\begin{proof}[Proof of Proposition~\ref{prop:convexHullR}]
	Consider the optimization of an arbitrary linear function over the convex relaxation of \rev{the extended formulation of} $U_R$ \rev{given by}:
	\begin{align}
	\ \ \ \ \ \ \ \min\; & a'x +b'y +p w+qz \notag\\
	\text{($P_R$)} \ \ \ \ \ \text{s.t.}\;& {s^2 +\sum_{i\in M}d_i y_i^2 +(w-z)^2}\leq (w+z)^2 \label{eq:rotatedCone}\\
	& (x,s)\in \text{conv}(\rev{U_0})\notag\\
	& y\in \R_+^m, w\geq 0, z\geq 0\notag.
	\end{align}
	
	Without loss of generality, we can assume that $p> 0$ and $q>0$ (if $p<0$ or $q<0$ then the problem is unbounded, and if $p=0$ or $q=0$ then ($P_R$) reduces to a linear program over an integral polyhedron). Moreover, observe that if $w=z$ in an optimal solution, then the problem reduces to a linear optimization over $\rev{\conv(}U\rev)$ which has an optimal integral solution (Proposition~\ref{prop:convexHullU}). Thus, we can assume that $w\neq z$, in which case the left hand size of \eqref{eq:rotatedCone} is differentiable, and we infer from KKT conditions with respect to $w$ and $z$ that 
	\begin{align}
	-p &= -\lambda +\lambda \frac{w-z}{\sqrt{s^2+\sum_{i\in M}d_iy_i^2+(w-z)^2}}\label{eq:dualFeasibilityW}\\
	-q &= -\lambda -\lambda \frac{w-z}{\sqrt{s^2+\sum_{i\in M}d_iy_i^2+(w-z)^2}}\label{eq:dualFeasibilityZ},
	\end{align}
	where $\lambda$ is the dual variable associated with constraint \eqref{eq:rotatedCone}. We deduce from \eqref{eq:dualFeasibilityW} that $w-z=\frac{\lambda-p}{\lambda}\sqrt{s^2+\sum\limits_{i\in M}d_iy_i^2+(w-z)^2}$, and from \eqref{eq:dualFeasibilityZ} that\begin{equation}w-z=\frac{q-\lambda}{\lambda}\sqrt{s^2+\sum\limits_{i\in M}d_iy_i^2+(w-z)^2}\label{eq:difference}.
	\end{equation} In particular, we find that $\lambda=\frac{p+q}{2}.$
	
	Moreover, we obtain from \eqref{eq:difference} that
	\begin{align*}
	(w-z)^2&=\left(\frac{q-\lambda}{\lambda}\right)^2\left(s^2+\sum_{i\in M}d_iy_i^2+(w-z)^2\right)\notag\\
	&=\left(\frac{q-p}{q+p}\right)^2\left(s^2+\sum_{i\in M}d_iy_i^2+(w-z)^2\right).
	\end{align*}
\rev{Letting $\beta=\frac{\left(\frac{q-p}{q+p}\right)^2}{1-\left(\frac{q-p}{q+p}\right)^2}$, we deduce that $$	(w-z)^2=\beta\left(s^2+\sum_{i\in M} d_iy_i^2\right).$$}
Therefore, we have that 
	\begin{equation*}
	\sqrt{s^2+\sum_{i\in M}d_iy_i^2+(w-z)^2}=\sqrt{1+\beta}\sqrt{s^2+\sum_{i\in M}d_iy_i^2}.\label{eq:sqrt}
	\end{equation*}
	Moreover, since in any optimal solution of ($P_R$) constraint \eqref{eq:rotatedCone} is binding, we have
	\begin{equation*}
	w+z=\sqrt{1+\beta}\sqrt{s^2+\sum_{i\in M}d_iy_i^2}.
	\label{eq:sum}
	\end{equation*}
	
	Multiplying equality \eqref{eq:dualFeasibilityW} by $w$ in both sides, and multiplying equality \eqref{eq:dualFeasibilityZ} by $z$ in both sides, we find that
	\begin{align}
	pw+qz&=\lambda(w+z)-\lambda \frac{(w-z)^2}{\sqrt{s^2+\sum_{i\in M}d_iy_i^2+(w-z)^2}}\notag\\
	&=\lambda\sqrt{1+\beta}\sqrt{s^2+\sum_{i\in M}d_iy_i^2}-\lambda \frac{\beta\left(s^2+\sum_{i\in M} d_iy_i^2\right)}{\sqrt{1+\beta}\sqrt{s^2+\sum_{i\in M}d_iy_i^2}}\notag \\
	&=\lambda\frac{s^2+\sum_{i\in M} d_iy_i^2}{\sqrt{1+\beta}\sqrt{s^2+\sum_{i\in M}d_iy_i^2}}\notag\\
	&=\frac{\lambda}{\sqrt{1+\beta}}\sqrt{s^2+\sum_{i\in M}d_iy_i^2}.\label{eq:substitution}
	\end{align}
	
	Therefore, \rev{substituting $pw+qz$ in the objective function of ($P_R$) by \eqref{eq:substitution} and using that $\lambda=\frac{p+q}{2}$,} we see that problem ($P_R$) reduces to 
	\leqnomode
	\begin{align}
	\ \ \ \ \ \ \ \min\; & a'x +b'y +\frac{p+q}{2\sqrt{1+\beta}}\sqrt{s^2+\sum_{i\in M}d_iy_i^2} \label{eq:PR2}\tag{$P_R'$}\\
	\text{s.t.}\;
	& (x,s)\in \text{conv}(U_0), y\in \R_+^m.\notag
	\end{align}
\rev{Moreover, \eqref{eq:PR2} is of the form of ($P_2$) in Proposition~\ref{prop:convexHullU} (after scaling), thus has an optimal integer solution. Therefore, after projecting out the additional variable $s$, we find the desired result.}
\end{proof}

\section{}
\label{sec:computationalBounded}

In this section we test the effectiveness of the \rev{\texttt{unbounded}} polymatroid inequalities \eqref{eq:polymatroidUnbounded} and \rev{\texttt{bounded}} inequalities \eqref{eq:polymatroidBounded} in solving optimization problems \rev{with bounded continuous variables} of the form \begin{equation}
\label{eq:computationalA}
\min \left \{-a'x -b'y +\Omega z: (x,y,z)\in H_\rev\G \right \}.
\end{equation}
For two numbers $\ell<u$,  let $U[\ell,u]$ denote the continuous uniform distribution between $\ell$ and $u$.
The data for the model is generated as follows: $a_i\sim U[0,1]$, $\sqrt{c_i}\sim U[0.85 a_i, 1.15 a_i]$ for $i\in N$, $b_j\sim U[0,1]$, $\sqrt{d_j}\sim U[0.85 b_j, 1.15 b_j]$ for $j\in M$, and $\Omega$ is the solution\footnote{This choice of $\Omega$ ensures that the linear and nonlinear components are well-balanced, resulting in challenging instances with large integrality gap.} of $$-a(N)-b(M)+\Omega\sqrt{c(N)+d(M)}=0. $$
These instances have large integrality gaps with a single conic quadratic constraint.

\rev{The \texttt{unbounded}} inequalities are added as linear cuts in an extended formulation, as described in \rev{Section~\ref{sec:implementationUnbounded}. The \texttt{bounded} inequalities are either added directly as nonlinear inequalities as described in Section~\ref{sec:implementationNonlinearCuts} (\texttt{bounded-nonlinear}), or using outer approximations as described in Section~\ref{sec:implementationGradient} (\texttt{bounded-gradient})}. 
A greedy heuristic is used to choose $T\subseteq M$ for inequalities  \eqref{eq:polymatroidBounded}: \rev{given a fractional point $(\bar x, \bar y, \bar z)$ with} $\bar{y}_{(1)}\geq \bar{y}_{(2)}\geq \ldots \geq \bar{y}_{(m)}$, we check for violation inequalities for each $T_i$ of the form $T_i=\{(1),(2),\cdots,(i)\}$ for $i=0,\ldots,m$. \rev{When using the nonlinear inequalities \texttt{bounded-nonlinear}, we iteratively solve the continuous relaxations and explicitly add the most violated inequality \eqref{eq:polymatroidBounded} found, and the process is repeated until the relative violation of the inequality found is less than $10^{-3}$, i.e., 
$$	
\frac{\sqrt{\left(\sqrt{\sigma+\sum_{i\in T}d_i\bar y_i^2}+\pi'\bar x\right)^2 +\sum_{i\in M\setminus T}d_i\bar y_i^2}}{\bar z}-1\leq 10^{-3}.$$
Observe that this process requires solving many continuous relaxations of \eqref{eq:computationalA} using the barrier algorithm (which is the default algorithm for convex conic quadratic optimization). For \texttt{bounded-gradient}, the inequalities are added at the root node of the branch-and-bound tree using CPLEX callbacks.
}


Table~\ref{tab:pOFFn100} presents the results. Each row represents the average over five instances generated with the same parameters and shows the number of discrete ($n$) and continuous ($m$) variables, the initial gap (\texttt{igap}), the root gap improvement (\texttt{rimp}), the number of nodes explored (\texttt{nodes}), the time elapsed \rev{(including the time used adding the inequalities)} in seconds (\texttt{time}), and the end gap (\texttt{egap})[in brackets, the number of instances solved to optimality (\#)]. The initial gap is computed as $\texttt{igap}=\frac{t_{\text{opt}}-t_{\text{relax}}}{\left|t_{\text{opt}}\right|}\times 100$, where $t_{\text{opt}}$ is the objective value of the best feasible solution at termination and $t_{\text{relax}}$ is the objective value of the continuous relaxation. The end gap is computed as $\texttt{egap}=\frac{{t_\text{opt}}-t_{\text{bb}}}{\left|t_{\text{opt}}\right|}\times 100$, where $t_{\text{bb}}$ is the objective value of the best lower bound at termination.
The root improvement is computed as $\texttt{rimp}=\frac{t_{\text{root}}-t_{\text{relax}}}{t_{\text{opt}}-t_{\text{relax}}}\times 100$, where $t_{\text{root}}$ is the value of the continuous relaxation after adding the valid inequalities to the formulation.

{
\renewcommand\arraystretch{0.8}
\begin{table}[h!]
\caption{Experiments with bounded continuous variables.}
\setlength{\tabcolsep}{2pt}
\begin{center}
\label{tab:pOFFn100}
\SingleSpacedXI
\scalebox{0.6}{
\begin{tabular}{ c c c |c  c c c| c c c c| c c c c | c c c c}
\hline \hline
\multirow{2}{*}{$n$} &\multirow{2}{*}{$m$} & \multirow{2}{*}{\textbf{\texttt{igap}}} & \multicolumn{4}{c|}{\textbf{\texttt{cpx}}} & \multicolumn{4}{c|}{\textbf{\texttt{unbounded}}}& \multicolumn{4}{c|}{\textbf{\texttt{bounded-gradient}}}&\multicolumn{4}{c}{\textbf{\texttt{bounded-nonlinear}}}\\
&&&\texttt{rimp}&\texttt{nodes}&\texttt{time}&\texttt{egap}[\#]&\texttt{rimp}&\texttt{nodes}&\texttt{time}&\texttt{egap}[\#]&\texttt{rimp}&\texttt{nodes}&\texttt{time}&\texttt{egap}[\#]&\texttt{rimp}&\texttt{nodes}&\texttt{time}&\texttt{egap}[\#]\\
\midrule
\multirow{3}{*}{100}&20 & 1,554.7 & 0.0 & 441,520 & 162 & 0.0[5] & 90.0 & 9,617 & 112 & 0.0[5] & 99.7 & 25 & 74 & 0.0[5] & 100.0 & 1 & 45 & 0.0[5]\\
&50 & 724.6 & 0.0 & 2,126,713 & 1,644 & 0.0[5] & 76.0 & 853,671 & 7,200 & 72.5[0] & 99.2 & 1,985 & 4,375 & 0.0[5]& 99.9 & 30 & 219 & 0.0[5]\\
&100 & 267.8 & 0.0 & 8,922,545 & 6,850 & 16.6[1] & 62.1 & 726,361 & 7,200 & 83.3[0] & 81.0 & - & 7,200 & 53.8[0] & 99.9 & 55 & 84 & 0.0[5]\\
\multicolumn{3}{c|}{\textbf{Average}} & \textbf{0.0} & \textbf{3,830,259} & \textbf{2,885} & \textbf{5.6[11]}& \textbf{76.0} & \textbf{529,883} & \textbf{4,804} & \textbf{51.9[5]}& \textbf{93.3} & \textbf{670} & \textbf{3,874} & \textbf{17.9[10]}& \textbf{100.0} & \textbf{29} & \textbf{116} & \textbf{0.0[15]}  \\
\hline
& &&&&&&&&&&&&&&&\\
\multirow{3}{*}{200}&40 & 987.1 & 0.0 & 15,133,028 & 7,200 & 352.7[0] & 89.3 & 127,408 & 7,200 & 72.9[0] & 99.5 & 85 & 6,253 & 3.5[3] & 100.0 & 52 & 475 & 0.0[5]\\
&100 & 396.6 & 0.0 & 11,650,607 & 7,200 & 397.3[0] & 73.9 & 57,742 & 7,200 & 100.7[0] & 79.7 & - & 7,200 & 133.2[0	]& 99.9 & 140 & 395 & 0.0[5]\\
&200 & 217.6 & 0.0 & 4,970,327 & 7,200 & 114.4[0] & 18.3 & 1,647,845 & 7,200 & 690.5[0] & 2.2 & 2,034,862 & 7,200 & 181.6[0] & 99.8 & 183 & 710 & 0.0[5]\\
\multicolumn{3}{c|}{\textbf{Average}} & \textbf{0.0} & \textbf{10,584,654} & \textbf{7,200} & \textbf{205.2[0]}& \textbf{60.5} & \textbf{610,998} & \textbf{7,200} & \textbf{213.1[0]}& \textbf{64.6} & \textbf{581,419} & \textbf{6,845} & \textbf{64.6[3]}& \textbf{99.9} & \textbf{125} & \textbf{527} & \textbf{0.0[15]}  \\
\hline\hline
\end{tabular}
}
\end{center}
\end{table}
} 

We observe in Table~\ref{tab:pOFFn100} that
the use \rev{of the \texttt{unbounded}} inequalities, which do not exploit the upper bounds of the continuous variables, closes 68.2\% of the initial gap on average, but the gap improvement does not necessarily translate to better solution times or end gaps. \rev{The performance of the \texttt{bounded} inequalities, when added as gradients, is adequate when $m$ is small, achieving close to 100\% root gap improvement. However, the performance degrades substantially as $m$ increases; in particular, for $m=100$, the full two hours are spent at the root node adding cuts, and the root improvement of close to 80\% is still far from 99.9\%, achieved by \texttt{bounded-nonlinear}. Moreover, for $n=200$ and $m=200$, both \texttt{unbounded} and \texttt{bounded-gradient} inequalities are ineffective at closing the root gap, with root improvements of 18.3\% and 2.2\%, respectively. In contrast, adding the \texttt{bounded} inequalities as nonlinear inequalities results in all cases in the best performance, with root improvements close to 100\%, significantly \rep{fewer}{less} branch-and-bound nodes explored and better solution times than the other formulations.}

\end{document}